\documentclass{scrartcl}
\usepackage{tikz}
\usepackage{amsthm}
\usepackage{amsfonts, amsmath, amssymb}

\newtheorem{theorem}{Theorem}[section]
\newtheorem{lemma}[theorem]{Lemma}
\newtheorem{proposition}[theorem]{Proposition}
\newtheorem{corollary}[theorem]{Corollary}
\theoremstyle{definition}
\theoremstyle{remark}
\newtheorem{remark}[theorem]{Remark}

\title{The 2-Cops-Move Cop Number of Graphs on the Torus, Klein Bottle, and Projective Plane}
\author{Florian Lehner and Alan Sun}
\date{\today}

\begin{document}

\maketitle

\begin{abstract}
We study the $2$-cops-move variant of the game of cops and robber, in which at most two cops may move in each turn. Let $c_2(G)$ denote the corresponding cop number. We prove that every graph embeddable on the torus, Klein bottle, or projective plane satisfies $c_2(G)\le 3$, extending a recent result of González Hermosillo de la Maza and Mohar for planar graphs.
\end{abstract}

\section{Introduction}
Cops and robber is a pursuit-evasion game played on a graph. In this perfect information game, a team of cops and a single robber occupy vertices of a graph and move alternately along edges. The cops win if one of them occupies the same vertex as the robber after finitely many moves; otherwise, the robber wins.
The game with a single cop was first studied independently by Nowakowski and Winkler~\cite{NowakowskiWinkler1983} and Quilliot~\cite{Quilliot1983}. The variant using multiple cops was introduced by Aigner and Fromme~\cite{AignerFromme1984} who defined the \emph{cop number} of a graph $G$, denoted $c(G)$, as the minimum number of cops needed to guarantee capture of the robber on $G$.

One of the central questions in this area concerns the relationship between the cop number and geometric and topological properties of a graph. Aigner and Fromme~\cite{AignerFromme1984} showed that every planar graph has cop number at most 3, and it is well known that there are graphs such as the dodecahedron net which meet this bound. This was extended to graphs of higher genus by Quilliot~\cite{Quilliot1985} who proved that any graph embeddable on a surface of genus $g$ has cop number at most $2g +3$. Schroeder~\cite{Schroeder2001} subsequently improved this bound to $\lfloor \frac{3g}{2}\rfloor + 3$ and conjectured a bound of $g+3$; the currently best known general upper bound on the cop number of graphs of genus $g$ is $\frac 43 g + \frac {10}3$, see~\cite{BowlerErde2021}. We refer the reader to the survey paper~\cite{BonatoMohar2020} by Bonato and Mohar for further background on topological directions in cops and robber.

A natural variant of cops and robber, introduced by Offner and Ojakian~\cite{OffnerOjakian2019} and studied under different names including \emph{lazy cops and robber}~\cite{BalBonato2016} and the \emph{one-cop-moves game}~\cite{YangHamilton2015}, restricts the cops' movements by allowing at most one cop to move in each turn. 

More generally, we can consider the \emph{$k$-cops-move game} where at most $k$ cops may move on each turn. Denoting by $c_k(G)$ the minimum number of cops required to win the $k$-cops-move game on graph $G$, it is clear that $c(G) \leq c_k(G)$. This raises the question how close this inequality is to being tight, and whether bounds for $c(G)$ carry over to the $k$-cops-move game. For the special case of planar graphs, Gao and Yang~\cite{GaoYang2017} provide examples with $c_1(G) \geq 4$, showing that Aigner and Fromme's bound does not extend to the lazy cop variant. On the other hand, González Hermosillo de la Maza and Mohar~\cite{GonzalezMohar2024} recently showed that $c_2(G) \leq 3$ for all planar graphs. This suggests that the 2-cops-move game behaves much more like the classical game than the lazy variant.

In this paper we consider the analogous problem for graphs embeddable on the torus, Klein bottle, or projective plane. For such graphs it is known that $c(G) \leq 3$; this was proved for the torus in~\cite{Lehner2021}, but the proof is easily seen to extend to the other two surfaces. The examples by Gao and Yang~\cite{GaoYang2017} show that $c_1(G) >3 $ is possible, leaving only $c_2$ open. We settle this remaining case by proving the following result.

\begin{theorem}
Let $G$ be a graph embeddable on the torus, Klein bottle, or projective plane. Then $c_2(G) \leq 3$.
\end{theorem}

Our strategy proceeds in phases, each of which decreases the robber territory, that is, the set of vertices accessible to the robber. It combines ideas from~\cite{GonzalezMohar2024}, which allow two cops to guard paths while only two cops move in each turn, with the cover-based approach from~\cite{Lehner2021}, which makes it possible to handle surfaces of higher genus.

As in~\cite{Lehner2021}, we work with an infinite cover of the graph and aim to restrict the robber to a finite region in this infinite cover which allows us to apply the result from the planar case. The main difference is that a priori we are not able to identify a finite set of `exit points' for the robber, so we need a different way of measuring progress towards our goal of restricting the robber to a finite subgraph. To this end, in Section \ref{subsec:periodicbands} we introduce the notion of \emph{bands} based on an equivalence relation on periodic geodesics. The number of bands in the robber territory essentially measures how many periodic geodesics the robber can access without being caught.

To ensure that only two cops move at any given turn, we guard paths based on ideas introduced by González Hermosillo de la Maza and Mohar~\cite{GonzalezMohar2024}. A key idea of~\cite{GonzalezMohar2024} is that, under suitable conditions on the robber's position, a cop can guard a geodesic while repeatedly remaining stationary for a turn. This property, which we refer to as leisurely guarding, enables two cops to guard paths while the third cop gets in position (moving only when one of the other cops remains stationary). 

Unfortunately, neither of the hypotheses required for the planar argument is automatically satisfied in our setting: the robber's position is not guaranteed to satisfy the necessary conditions, and the geodesics that must be guarded are generally infinite. The first difficulty is resolved by refining the leisurely-guarding machinery from~\cite{GonzalezMohar2024}; see Section~\ref{subsec:wideshadow}, in particular Lemma~\ref{lem:leisurelyguarding}. The second is overcome by ensuring that each phase can succeed in one of two ways: either leisurely guarding eventually applies when the robber remains bounded (allowing the third cop to move), or the structure of the guarded paths forces progress when the robber attempts to escape to infinity.

These two mechanisms ensure that every phase of the strategy eventually decreases the number of bands the robber territory. Iterating this argument reduces the robber territory to a finite region, where capture follows from the planar techniques of~\cite{GonzalezMohar2024}.

\section{Preliminaries}
Graphs in this paper may be infinite, but are always assumed to be locally finite (that is, every vertex has finitely many neighbours), simple, and connected. 

\subsection{Paths and Geodesics}

A \emph{path} $P$ is a sequence of vertices $P = (v_i)_{i \in I}$ such that $\{v_{i-1}, v_i\} \in E(G)$ and all vertices $v_i$ are distinct, where $I \subseteq \mathbb Z$ is a set of consecutive integers, that is, if $i,j \in I$ and $i \leq k \leq j$, then $k \in I$. The \emph{length} of a path $P$ is $|I|-1$, and is denoted by $|P|$. The first and last vertex in a path (if they exist) are called its endpoints. A path with endpoints $x$ and $y$ is said to \emph{connect} $x$ and $y$. A path with exactly one endpoint is called a \emph{ray}, and a path with no endpoints is called a \emph{double ray}.

The \emph{distance} between two vertices $x,y \in V(G)$, denoted by $d(x,y)$, is the minimal length of a path connecting $x$ and $y$. Note that this is well-defined since graphs are assumed to be connected.
A path $P$ is called a \emph{geodesic} if $d(v_i,v_j) = i-j$ for any pair of vertices on $P$. 

We will frequently restrict and concatenate paths; the following notation will be useful. Let $P$ be a path and let $v$ be a vertex of $P$. We denote by $vP$ the subpath of $P$ starting at $v$, that is, the path consisting of $v$ and all vertices after $v$ on $P$. Similarly, we denote by $Pv$ the subpath of $P$ ending at $v$. Combining these two notations, $uPv$ is the subpath of $P$ from $u$ to $v$. We also occasionally define a path $P$ connecting $u$ and $v$ as $uPv$ if we want to emphasize the endvertices of the path. If $P$ and $Q$ are two paths, and $v$ is contained in both, we denote by $PvQ$ the concatenation of $Pv$ and $vQ$, in other words, $PvQ$ is the path following $P$ until $v$, and $Q$ from $v$ onwards. We say that $P$ \emph{eventually follows} $Q$ if there is some vertex $v$ (which necessarily lies in both $P$ and $Q$) such that $P = PvQ$.

The following definition is adapted from~\cite{GonzalezMohar2024}. Let $G$ be a graph, let $H$ be a subgraph of $G$, and $P$ be a geodesic in $H$. A path $uBv$ of finite length in $H$ is called a \emph{bypath} of $P$ in $H$ if $u$ and $v$ are contained in $P$ and $PuBvP$ is a geodesic in $H$, see Figure \ref{fig:bypath}. Note that bypaths of the form $uPv$ always exist, so we call these bypaths trivial. A geodesic $P$ is \emph{bypath-free} in $H$ if $H$ contains no non-trivial bypath of $P$. If $P$ is bypath-free, then every subpath of $P$ is bypath-free. Moreover, paths of length 1 are trivially bypath-free.

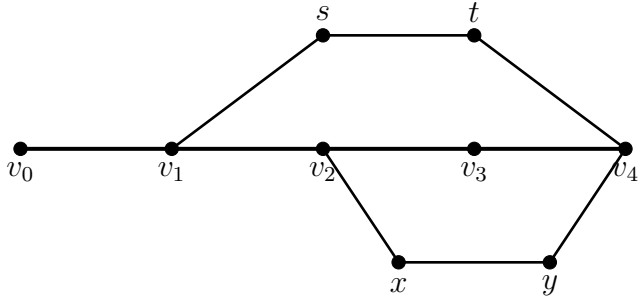
\begin{figure}
    \centering
    \begin{tikzpicture}[line cap=round,line join=round]
\clip(-0.5,-2.2) rectangle (8.5,2.5);
\draw [line width=1.5pt] (0,0)-- (8,0);
\draw [line width=1pt] (4,0)-- (5,-1.5);
\draw [line width=1pt] (5,-1.5)-- (7,-1.5);
\draw [line width=1pt] (7,-1.5)-- (8,0);
\draw [line width=1pt] (2,0)-- (4,1.5);
\draw [line width=1pt] (4,1.5)-- (6,1.5);
\draw [line width=1pt] (6,1.5)-- (8,0);
\begin{scriptsize}
\draw [fill=black] (0,0) circle (2.5pt);
\draw[color=black] (0,-0.3) node[font=\large] {$v_{0}$};
\draw [fill=black] (2,0) circle (2.5pt);
\draw[color=black] (2,-0.3) node[font=\large] {$v_{1}$};
\draw [fill=black] (4,0) circle (2.5pt);
\draw[color=black] (4,-0.3) node[font=\large] {$v_{2}$};
\draw [fill=black] (6,0) circle (2.5pt);
\draw[color=black] (6,-0.3) node[font=\large] {$v_{3}$};
\draw [fill=black] (8,0) circle (2.5pt);
\draw[color=black] (8,-0.3) node[font=\large] {$v_{4}$};
\draw [fill=black] (5,-1.5) circle (2.5pt);
\draw[color=black] (5,-1.8) node[font=\large] {$x$};
\draw [fill=black] (7,-1.5) circle (2.5pt);
\draw[color=black] (7,-1.8) node[font=\large] {$y$};
\draw [fill=black] (4,1.5) circle (2.5pt);
\draw[color=black] (4,1.8) node[font=\large] {$s$};
\draw [fill=black] (6,1.5) circle (2.5pt);
\draw[color=black] (6,1.8) node[font=\large] {$t$};
\end{scriptsize}
\end{tikzpicture}
    \caption{\small Path $B = Pv_1D$ where $D = v_1stv_4$ is a bypath of $P = v_0v_1v_2v_3v_4$ but path $C = Pv_2E$ where $E = v_2xyv_4$ is not.}
    \label{fig:bypath}
\end{figure}

\subsection{Periodic Covers}

An \emph{automorphism} of a graph $G$ is a bijection $\phi \colon V(G) \to V(G)$ such that two vertices $v$ and $w$ are adjacent if and only if $\phi(v)$ and $\phi(w)$ are adjacent. 
A graph $\hat{G} = (\hat{V}, \hat{E})$ is a \emph{cover} of $G = (V,E)$ if there exists a surjective map $\phi \colon \hat{V} \to V$ such that for every $v \in \hat{V}$, $\phi$ is a bijection from $N(v)$ to $N(\phi(v))$. The map $\phi$ is called a \emph{covering map}. 

\begin{lemma}
    \label{lem:cover}
    Let $G$ be a finite graph embedded on a torus, Klein bottle, or projective plane. Then $G$ has a cover which can be embedded in an infinite cylinder $\mathbb R \times \mathbb R/\mathbb Z$ so that the shift $\sigma \colon (x,y) \mapsto (x+1,y)$ induces an automorphism of the cover, with a covering map $\phi$ satisfying $\phi \circ \sigma = \phi$. Moreover, only finitely many vertices embed into each annulus $[x,x+1) \times R/\mathbb Z$.
\end{lemma}

\begin{proof}
    Recall that the universal cover of each of these three surfaces is $\mathbb R^2$, and that the fundamental polygon in each of these covers may be chosen to be the unit square $[0,1)^2$. 
    
    This gives rise to a planar cover $G'$ of any graph $G$ embedded on one of the 3 surfaces. 
    Let $\iota \colon V(G') \to \mathbb R^2$ be the map that takes every vertex to its image under the embedding obtained from this covering.

    Regardless of which of the three cases we are in, given any pair of integers, the map $\sigma_{m,n}\colon (x,y) \mapsto (x+2m,y+2n)$ preserves the embedding and thus induces an automorphism of $G'$.
    Note that the covering map $\phi\colon G' \to G$ in each of the three cases satisfies $\phi \circ \sigma_{m,n} = \phi$.

    Define an equivalence relation on vertices of $G'$ by $v \sim w$ if there is some $n \in \mathbb Z$ such that $\sigma_{0,n}(\iota(v)) = \iota(w)$, in other words, if the embeddings of $v$ and $w$ have the same $x$-coordinate, and their $y$-coordinates differ by a multiple of $2$. Define an auxiliary graph $G''$ whose vertices are the equivalence classes with respect to the relation $\sim$ with an edge between two of these equivalence classes if there is an edge connecting two elements contained in the two classes, respectively.

    We note that $G''$ is embedded in an infinite cylinder surface $\mathbb R^2/ (\{0 \times 2\mathbb Z\})$. Moreover, if $v \sim w$, then by the above observation $\phi(v) = \phi(w)$, and thus the map $\phi$ gives rise to a map $\psi\colon G'' \to G$. Since $\phi$ is a covering map, so is $\psi$.

    Finally, scaling the cylinder by $\frac 12$ we obtain the desired embedding.
\end{proof}

Throughout this paper, when considering covers as in the above lemma (or graphs as in the following remark), we call a double ray \emph{periodic} if it is invariant under some power $\sigma^k \colon (x,y) \mapsto (x+k,y)$  of the shift $\sigma$.

\begin{remark}
    \label{rmk:stripembedding}
    Let $G$ be one of the covers constructed in Lemma \ref{lem:cover} and let $P$ be a periodic geodesic. Let $H$ be obtained by `splitting $P$ in $G$'. More precisely, $H$ is obtained from $G$ as follows: First, replace every vertex $v_i$ of $P$ by two vertices $v_i^1$ and $v_i^2$. Connect $v_i^1$ to all neighbours of $v_i$ which come between $v_{i-1}$ and $v_{i+1}$ in the clockwise cyclic ordering induced by the embedding on the cylinder, and connect $v_i^2$ to all neighbours of $v_i$ which come between $v_{i+1}$ and $v_{i-1}$ in the clockwise cyclic ordering induced by the embedding. In other words, $v_i^1$ is connected to all neighbours of $v_i$ which are attached on one side of $P$ and $v_i^2$ is connected to all neighbours of $v_i$ which are attached on the other side.
    
    Then $H$ can be embedded in an infinite strip $\mathbb R \times [0,1]$ such that the two copies of $P$ obtained in the splitting process are embedded at the top and bottom of the strip, and (possibly after re-scaling) the shift $(x,y) \mapsto (x+1,y)$ induces an automorphism of $H$.
\end{remark}

\subsection{Cops and Robber Variants}

We briefly recall some variants of the cops and robber game which are relevant for this paper.

Cops and robber is a perfect information pursuit-evasion game played on  a finite connected undirected graph $G$. It is played between two players, one controlling a set of \emph{cops} $c_1, c_2, c_3, \dots$ and the other controlling a \emph{robber} $r$. By slight abuse of notation, $c_i$ refers both to the $i$-th cop and to the vertex of the graph this cop currently occupies. In particular, $c_i = x$ means that the current position of $c_i$ is the vertex $x$, and $c_i$ moves from $x$ to $y$ means that the position of the cop $c_i$ was $x$ before this move, and is $y$ afterwards. Occasionally, when we want to compare positions before and after a move, we denote positions before and after the move by $c_i$ and $c_i'$, respectively; this will always be explicitly stated. We follow analogous conventions for the robber.

The rules of the game are as follows: First, the cop player chooses starting positions for all cops, then the robber chooses a starting position. After that, the players alternate in taking moves. On each cop move, each cop $c_i$ may (but does not have to) move from its current position to an adjacent vertex, on a robber move $r$ may move to an adjacent vertex. The cops win if after a finite number of turns $c_i=r$ for some $i$ (we say that \emph{$c_i$ catches $r$}). The robber wins if this never happens. Call a graph $k$-cop win, if there is a winning strategy for $k$ cops. The cop number $c(G)$ of a graph $G$ is the least $k$ for which $G$ is $k$-cop win.

The \emph{$l$-cops move game} is a variant of the game described above where on each turn for the cops at most $l$ cops may move to adjacent vertices and the remaining cops stay at their current positions. The robber moves exactly as it would in the classical game of cops and robber. In the extreme case where $l=1$ this is known as \emph{lazy-cops and robber}. As above, we can define the \emph{$l$-cops move cop number $c_l(G)$} of a graph $G$ as the least number of cops for which a winning strategy on $G$ exists. Clearly, if $l \leq l'$, then the $l$-cops move cop number is greater than or equal to the $l'$-cops move cop number.

In~\cite{Lehner2021}, the following ``teleportation'' variant of the game was introduced to compare cop numbers of graphs with cop numbers of their respective covers. Let $G$ be a graph, and let $\hat G$ be a cover of $G$. Denote the covering map by $\phi$. Then \emph{T-cops and robber} is defined exactly like cops and robber except that on the cops' move, a cop located at some vertex $v$ can not only move to any vertex in the closed neighbourhood of $v$, but to any vertex in the closed neighbourhood of any $w$ with $\phi(w) = \phi(v)$; intuitively this means that a cop can \emph{teleport} to a vertex with the same image under $\psi$ before moving. For the purpose of the 2-cops move teleportation game, only teleporting from $v$ to any $w$ with $\phi(w) = \phi(v)$ does not count as having moved (since the projection of the respective cop's position in the original graph does not change). 

The following lemma can be proved in the exact same way as~\cite[Lemma 3.2]{Lehner2021}

\begin{lemma}
    \label{lem:teleportation}
    Let $G$ be a graph and let $\hat G$ be a cover of $G$. If $k$ cops can win the T-cops and robber game on $\hat G$ with at most $l$ cops moving at any given turn, then $k$ cops can win the (regular) cops and robber game on $G$ with at most $l$ cops moving at any given turn. 
\end{lemma}

Given a group $\Gamma$ acting on a graph $G$, we can define a game similar to T-cops and robber by allowing a cop located at some vertex $c$ to move to any vertex in the closed neighbourhood of any $c'$ for which there is some $\gamma \in \Gamma$ with $c' = \gamma (c)$. We call this game \emph{$\Gamma$-cops and robber}. If $\Gamma$ is generated by a single element $\sigma$, we write $\sigma$-cops and robber rather than $\langle \sigma\rangle$-cops and robber.

Now let $\hat G$ be the cover from Lemma~\ref{lem:cover} and let $\Gamma$ be the group generated by the shift $\sigma$. Since $\phi \circ \sigma = \phi$, all moves which are available to the cops in $\Gamma$-cops and robber are also available in T-cops and robber. Thus any winning strategy for the cops in the former is also a winning strategy for the latter. This is still true if we restrict to the $l$-cops-move game. Together with Lemma \ref{lem:teleportation} above, this implies the following.

\begin{lemma}
    \label{lem:sigmateleportation}
    Let $G$ be a graph embedded on the torus, projective plane or Klein bottle, let $\hat G$ be the cover from Lemma~\ref{lem:cover}, and let $\sigma$ be the shift automorphism. If $k$ cops can win the $\sigma$-cops and robber game on $\hat G$ with at most $l$ cops moving at any given turn, then $k$ cops can win the (regular) cops and robber game on $G$ with at most $l$ cops moving at any given turn. 
\end{lemma}

\section{Periodic Geodesics and Bands}
\label{subsec:periodicbands}

Throughout this section, let $G$ be one of the covers from Lemma \ref{lem:cover} or obtained from splitting such a cover along a periodic geodesic as in Remark \ref{rmk:stripembedding}. This graph $G$ embeds either on the infinite cylinder $\mathbb R \times \mathbb R/\mathbb Z$ or on the infinite strip $\mathbb R \times [0,1]$. Since both surfaces are orientable, we may fix an embedding, thereby obtaining a (clockwise) cyclic order of the edges at each vertex. By convention, we refer to the directions decreasing or increasing the first coordinate as left and right, and the directions decreasing or increasing the second coordinate (in particular for the strip, but also locally in the cylinder) as down and up. We refer to a fundamental domain $\mathcal{C}_k := [k,k+1) \times X$ where $X$ is either $\mathbb R/\mathbb Z$ or $[0,1]$ as a \emph{cell} of the embedding. 

It is known that $G$ contains periodic geodesics, see~\cite{PolatWatkins1995}. Note that any geodesic intersects each cell in only finitely many vertices. In particular, any geodesic double ray $P = (v_i)_{i \in \mathbb Z}$ has the property that as $i\to \infty$ the first coordinate of the embedding of $v_i$ either goes to $\infty$ or to $-\infty$. In the first case we say that \emph{$P$ tends to the right}, in the second case we say that \emph{$P$ tends to the left}. All periodic geodesics are tacitly assumed to tend to the right.

Define an equivalence relation on the set of periodic geodesics by $P \sim Q$ if there is a bypath of $P$ containing a vertex of $Q$. We show that this is indeed an equivalence relation.

\begin{proposition}
    The relation $\sim$ is an equivalence relation.
\end{proposition}
\begin{proof} For reflexivity, note that each vertex of a periodic geodesic $P$ is a (trivial) bypath to $P$, so $P \sim P$. 

We now prove symmetry, see Figure\ref{fig:periodic-geodesic-equivalence-symmetry} for a sketch. Consider periodic geodesics such that $P \sim Q$. Let $xBy$ be a bypath of $P$ which meets $Q$, and let $z$ be a vertex of $B$ on $Q$. Because $P$ and $Q$ are periodic and $B$ is finite, there is some $k$ such that the shift $\sigma^k$ preserves both $P$ and $Q$, and $\sigma^k(x)$ comes after $y$ on $P$. Since $B$ and $\sigma(B)$ are both bypaths of $P$, the double ray $PxByP\sigma^k(x)\sigma^k(B)\sigma^k(y)P$ is a geodesic. Thus $B' = zByP\sigma^k(x)\sigma^k(B)\sigma^k(z)$ is a geodesic from $z$ to $\sigma^k(z)$, and therefore it has the same length as $zQ\sigma^k(z)$. Therefore $B'$ is a bypath of $Q$. Since $B'$ contains a vertex of $P$, this shows that $Q \sim P$.

\begin{figure}
\centering
\begin{tikzpicture}[every node/.style={font=\small}]

\draw (0,0) -- (13,0);
\draw (0,1.6) -- (13,1.6);

\node[left] at (0,0) {$P$};
\node[left] at (0,1.6) {$Q$};

\draw[thick] (1.8,0)
  .. controls (2.1,0.9) and (2.6,1.6) .. (3.0,1.6)
  .. controls (3.4,1.6) and (3.9,0.9) .. (4.2,0);
\node at (3.0,0.75) {$B_1$};
\node at (1.8,-0.5) {$x$};
\node at (4.2,-0.5) {$y$};

\draw[thick] (9.0,0)
  .. controls (9.3,0.9) and (9.8,1.6) .. (10.2,1.6)
  .. controls (10.6,1.6) and (11.1,0.9) .. (11.4,0);
\node at (10.3,0.75) {$\sigma^k(B_1)$};
\node at (9,-0.5) {$\sigma^k(x)$};
\node at (11.5,-0.5) {$\sigma^k(y)$};

\fill (1.8,0) circle (1.5pt);
\fill (4.2,0) circle (1.5pt);
\fill (3.0,1.6) circle (1.5pt) node[above=2pt] {$z_1$};
\fill (9.0,0) circle (1.5pt);
\fill (11.4,0) circle (1.5pt);
\fill (10.2,1.6) circle (1.5pt) node[above=2pt] {$\sigma^k(z_1)$};

\draw[thick, dashed] (3.0,1.7)
  .. controls (3.4,1.7) and (3.9,1.3) .. (4.3,0.15);

\draw[thick, dashed] (4.3,0.15) -- (8.9,0.15);

\draw[thick,dashed] (8.9, 0.15)
  .. controls (9.3,1.1) and (9.8,1.8) .. (10.2,1.7);

\node at (6.5, 0.6) {$B'$};

\end{tikzpicture}
\caption{Sketch of the proof of symmetry of the relation $\sim$. The dashed path is a bypath of $Q$ which meets $P$.}
\label{fig:periodic-geodesic-equivalence-symmetry}
\end{figure}
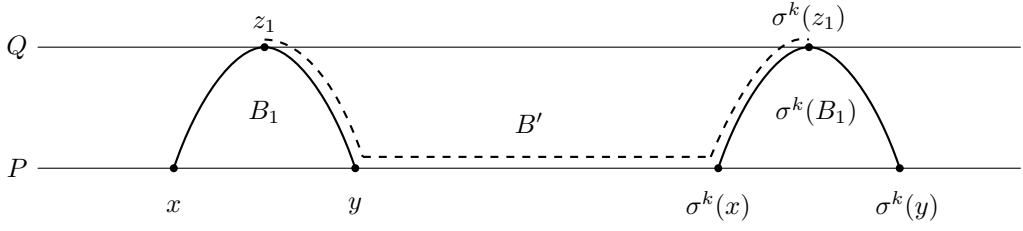

Finally, the proof of transitivity is sketched in Figure \ref{fig:periodic-geodesic-equivalence-transitivity}. Consider periodic geodesics, $P$, $Q$, and $R$ such that $P \sim Q$ and $Q \sim R$. Let $x_1B_1y_1$ be a bypath of $P$ containing a vertex $z_1$ of $Q$ and let $x_2B_2y_2$ be a bypath of $Q$ containing a vertex $z_2$ of $R$. Without loss of generality assume that $x_2$ comes after $z_1$ on $Q$, otherwise replace $B_2$ by $\sigma^k(B_2)$ for suitable $k$.
As before, we can find some $k$ such that $\sigma^k$ preserves $P$, $Q$ and $R$, and so that $\sigma^k(x_1)$ comes after $y_1$ on $P$ and $\sigma^k(z_1)$ comes after $y_2$ on $Q$. It is not hard to see that $B' = x_1B_1z_1Qx_2B_2y_2Q\sigma^k(z_1)\sigma^k(B_1)\sigma^k(y_1)$ is a bypath of $P$ which contains a vertex of $R$.
\begin{figure}
\centering
\begin{tikzpicture}[every node/.style={font=\small}]

\draw (0,0) -- (13,0);
\draw (0,1.6) -- (13,1.6);
\draw (0,3.2) -- (13,3.2);

\node[left] at (0,0) {$P$};
\node[left] at (0,1.6) {$Q$};
\node[left] at (0,3.2) {$R$};

\draw[thick] (1.8,0)
  .. controls (2.1,0.9) and (2.6,1.6) .. (3.0,1.6)
  .. controls (3.4,1.6) and (3.9,0.9) .. (4.2,0);
\node at (3.0,0.75) {$B_1$};

\draw[thick] (5.0,1.6)
  .. controls (5.4,2.5) and (6.0,3.2) .. (6.5,3.2)
  .. controls (7.0,3.2) and (7.6,2.5) .. (8.0,1.6);
\node at (6.5,2.4) {$B_2$};

\draw[thick] (9.0,0)
  .. controls (9.3,0.9) and (9.8,1.6) .. (10.2,1.6)
  .. controls (10.6,1.6) and (11.1,0.9) .. (11.4,0);
\node at (10.3,0.75) {$\sigma^k(B_1)$};

\fill (1.8,0) circle (1.5pt);
\fill (4.2,0) circle (1.5pt);
\fill (3.0,1.6) circle (1.5pt) node[above=2pt] {$z_1$};
\fill (5.0,1.6) circle (1.5pt);
\fill (8.0,1.6) circle (1.5pt);
\fill (6.5,3.2) circle (1.5pt) node[above=2pt] {$z_2$};
\fill (9.0,0) circle (1.5pt);
\fill (11.4,0) circle (1.5pt);
\fill (10.2,1.6) circle (1.5pt) node[above=2pt] {$\sigma^k(z_1)$};

\draw[thick, dashed] (0,0.15) -- (1.7,0.15);
\draw[thick, dashed] (3,1.75) -- (4.9,1.75);
\draw[thick, dashed] (8.1,1.75) -- (10.2,1.75);
\draw[thick, dashed] (11.5,0.15) -- (13,0.15);

\draw[thick,dashed] (1.75,0.15)
  .. controls (2.1,1.1) and (2.6,1.8) .. (3.0,1.7);
\draw[thick,dashed] (4.95,1.75)
  .. controls (5.4,2.7) and (6.0,3.4) .. (6.5,3.3)
  .. controls (7.0,3.3) and (7.6,2.7) .. (8.05,1.7);
\draw[thick,dashed] (10.2,1.7)
  .. controls (10.6,1.8) and (11.1,1.1) .. (11.5,0.1);
\node at (0.7,0.6) {$B'$};

\end{tikzpicture}
\caption{Sketch of the proof of transitivity of the relation $\sim$. The dashed path is a bypath of $P$ which meets $R$.}
\label{fig:periodic-geodesic-equivalence-transitivity}
\end{figure}
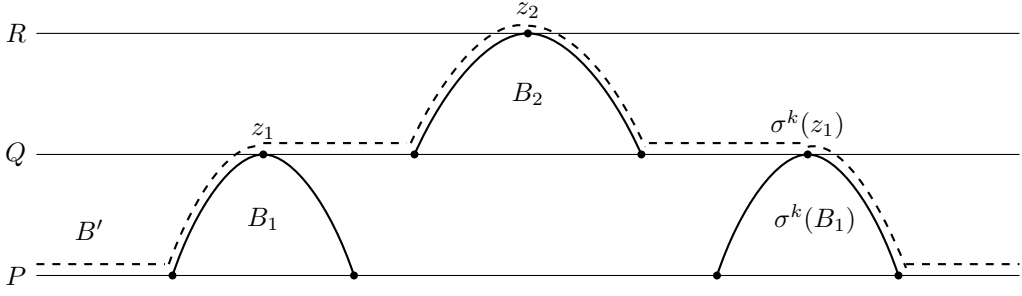
\end{proof}

Let $\mathcal C$ be an equivalence class with respect to the equivalence relation $\sim$. The \emph{band}, $\mathcal{B}$, corresponding to $\mathcal C$ is defined as the set of all vertices which lie on some geodesic contained in $\mathcal C$. If $P \in \mathcal{C}$, we also call $\mathcal B$ the band of $P$. In other words, the band of $P$ is the union of all periodic geodesics which contain a vertex of some bypath of $P$.

\begin{proposition}\label{Prop2.3}
    \begin{enumerate}
        \item Each band is connected, and distinct bands are disjoint.
        \item Every band is periodic, that is, there is some $k$ such that $\mathcal B = \sigma^k(\mathcal B)$.
    \end{enumerate}
\end{proposition}

\begin{proof}
    In each band, all pairs of vertices $x,y$ lie on geodesics. If they lie on the same geodesic, then they are connected by a finite geodesic subpath of that geodesic. If they are in distinct geodesics, then these geodesics are connected by a bypath and all vertices on this bypath are necessarily contained in the band since we can construct another periodic geodesic by replacing segments of one geodesic by suitable shifts of this bypath.

    If two distinct bands are not disjoint, then they contain geodesics which intersect. Any vertex in the intersection of these geodesics is a (trivial) bypath of one of them containing a vertex of the other, showing that they are equivalent, a contradiction.

    For the second part, note that if $\mathcal B$ is a band, then $\sigma(\mathcal B)$ is also a band. Further note that there is some $k$ such that $\mathcal B \cap \sigma^k(\mathcal B)$ is non-empty (take for instance the period of some geodesic in the corresponding equivalence class). 
\end{proof}

\begin{lemma}
    \label{lem:ray-q-to-P}
    Given a periodic geodesic $P= (v_i)_{i \in \mathbb Z}$ and a vertex $v \in V(G)$, there is a geodesic ray $Q$ and an integer $k$ such that $Q$ starts at $v$ and eventually follows $P$. 
\end{lemma}

\begin{proof}
    Note that $d(v,v_{i+1}) \leq d(v,v_{i}) + 1$ for every $i \in \mathbb N$.
    If there is some $k \in \mathbb N$ such that equality holds for each $i \geq k$, then $d(v,v_i) = d(v,v_k) + (i-k)$ for all $i \geq k$, and therefore the ray obtained by first taking a geodesic from $v$ to $v_k$ and then following $P$ is a geodesic with the desired properties.

    Hence assume that there is no such $k$. This implies that there are infinitely many $i$ such that $d(v,v_{i+1}) \leq d(v,v_{i})$, and thus for every $j \in \mathbb N$ there is some $i$ for which $d(v,v_{i}) < d(v,v_0) + i - j$. In particular, there is some $i \in \mathbb N$ for which $d(v,v_{i}) < d(v,v_0) + i - 2 d(v,v_0)$ and thus $i > d(v,v_{i}) + d(v,v_0)$. But then $v_0Pv_i$ would not be a shortest path, contradicting the assumption that $P$ is a geodesic.
\end{proof}

\begin{lemma}
    \label{lem:samespeedgeodesics}
    Let $P_1$ and $P_2$ be periodic geodesics, and let $v_1 \in P_1$ and $v_2 \in P_2$. If $\sigma^k$ preserves both $P_1$ and $P_2$, then $d(v_1,\sigma^k(v_1)) = d(v_2,\sigma^k(v_2))$.
\end{lemma}

\begin{proof}
    We first observe that  $d(v_1,\sigma^{ik}(v_1)) = i d(v_1,\sigma^k(v_1))$ for every $i$ because all vertices $\sigma^{ik}(v_1)$ lie on the periodic geodesic $P_1$, and similarly $d(v_2,\sigma^{ik}(v_2)) = i d(v_2,\sigma^k(v_2))$. We hence obtain
    \begin{multline*}
        i d(v_1,\sigma^k(v_1)) = d(v_1,\sigma^{ki}(v_1)) \leq d(v_1,v_2) + d(v_2,\sigma^{ki}(v_2)) + d(\sigma^{ki}(v_2),\sigma^{ki}(v_1))\\\
        = 2 d(v_1,v_2) + d(v_2,\sigma^{ki}(v_2)) = 2 d(v_1,v_2) + i d(v_2,\sigma^k(v_2)).
    \end{multline*}
    Since distances in graphs are integer valued, if $i > 2 d(v_1,v_2)$, then this implies that $d(v_1,\sigma^k(v_1)) \leq d(v_2,\sigma^k(v_2))$. An analogous argument shows the converse inequality.
\end{proof}

The next lemma provides a tool for constructing geodesics connecting different bands, see Figure \ref{fig:switchgeodesic} for a sketch.
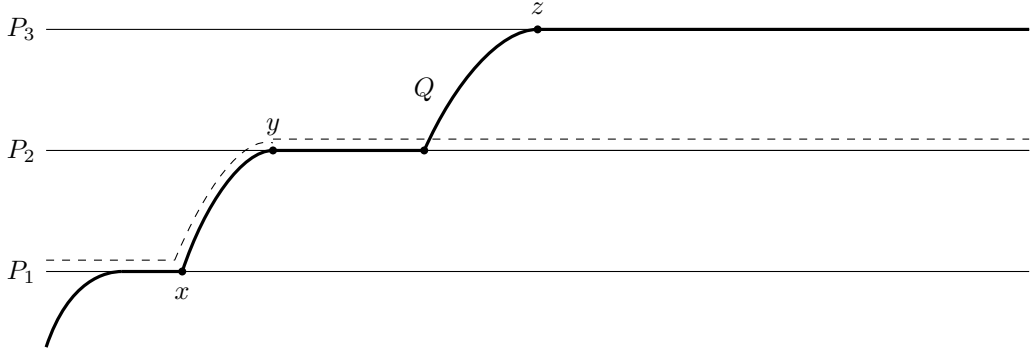
\begin{figure}
\centering
\begin{tikzpicture}[every node/.style={font=\small}]

\draw (0,0) -- (13,0);
\draw (0,1.6) -- (13,1.6);
\draw (0,3.2) -- (13,3.2);

\node[left] at (0,0) {$P_1$};
\node[left] at (0,1.6) {$P_2$};
\node[left] at (0,3.2) {$P_3$};

\fill (1.8,0) circle (1.5pt) node[below=2pt] {$x$};
\fill (3.0,1.6) circle (1.5pt) node[above=2pt] {$y$};
\fill (6.5,3.2) circle (1.5pt) node[above=2pt] {$z$};

\draw[very thick] (0,-1)
  .. controls (0.3,-0.1) and (0.8,0) .. (1,0);
\draw[very thick] (1,0) -- (1.8,0);
\draw[very thick] (3,1.6) -- (5,1.6);
\draw (5,1.6) -- (13,1.6);
\fill (5.0,1.6) circle (1.5pt);

\draw[very thick] (1.8,0)
  .. controls (2.1,0.9) and (2.6,1.6) .. (3.0,1.6);
\draw[very thick] (5,1.6)
  .. controls (5.4,2.5) and (6.0,3.2) .. (6.5,3.2);
\draw[very thick] (6.5,3.2) -- (13,3.2);
\node at (5,2.4) {$Q$};

\draw[dashed] (0,0.15) -- (1.7,0.15);
\draw[dashed] (1.7, 0.15)
  .. controls (2.1,1.1) and (2.6,1.8) .. (3.0,1.7);
\draw[dashed] (3.0,1.75) -- (13,1.75);

\end{tikzpicture}
\caption{The setup of Lemma~\ref{lem:switchgeodesic}: $P_1$, $P_2$, and $P_3$ are periodic geodesics, and $Q$ is a geodesic containing vertices $x\in P_1$ and $y\in P_2$ and eventually following $P_3$. The lemma states that the dashed line is a geodesic.}
\label{fig:switchgeodesic}
\end{figure}

\begin{lemma}
\label{lem:switchgeodesic}
    Let $P_1$, $P_2$, and $P_3$ be periodic geodesics. Let $Q$ be a geodesic eventually following $P_3$ which contains a vertex $x$ on $P_1$ and $y$ on $P_2$. Then $P_1xQyP_2$ is a geodesic.
\end{lemma}

\begin{proof}
Enumerate the vertices of $xQyP_2$ as $(x = v_0, v_1, v_2, \dots)$. Let $p$ be any vertex of $P_1$ which precedes $x$.

We claim that $d(p,v_i) = d(p,x)+i$. Since $P_1xQv_i$ or $P_1xQyP_2v_i$ is a path of this length, the only way this could fail is if $d(p,v_i) < d(p,x)+i$. In this case, let $S$ be a geodesic from $p$ to $v_i$. Replacing the subpath $P_1xQyP_2$ from $p$ to $v_i$ by $S$ we see that $d(p,v_j) < d(p,v_0)+j$ for all $j \geq i$, so we may without loss of generality assume that $v_i \in P_2$. Let $k$ be such that $\sigma^k$ preserves $P_1$, $P_2$, and $P_3$, and $\sigma^k(p)$ comes after $x$ on $P_1$ and $\sigma^k(y)$ comes after $v_i$ on $P_2$. Let $z$ be a vertex on $Q$ after which $Q$ follows $P_3$. Using Lemma \ref{lem:samespeedgeodesics}, and the fact that $P_1$, $P_2$, and $P_3$ are geodesics we compute
\[
d(z,\sigma^k(z)) = d(p,\sigma^k(p)) = d(p,x) + d(x,\sigma^k(p)),
\]
and 
\[
    d(z,\sigma^k(z)) = d(y,\sigma^k(y))  = d(y,v_i) + d(v_i,\sigma^k(y)).
\]
Moreover, $d(x,y) + d(y,v_i) = i$ since $P_2$ and $Q$ are geodesics. Now, using the inequality $d(p,v_i) < d(p,x)+i$ we obtain
\begin{align*}
    d(x,\sigma^{2k}(z)) 
    &= d(x,y) + d(y,z) + d(z,\sigma^k(z)) + d(z,\sigma^k(z))\\
    &= d(x,y) + d(y,z) + d(x,\sigma^k(p)) + d(p,x) + d(v_i,\sigma^k(y))+ d(y,v_i)\\
    &> i + (d(p,v_i) - i) +  d(y,z) + d(x,\sigma^k(p)) + d(v_i,\sigma^k(y))\\
    & = d(x,\sigma^k(p)) + d(p,v_i) + d(v_i,\sigma^k(y)) + d(y,z).
\end{align*}
The last expression in this computation is the length of the path
\[xP_1\sigma^k(p) \sigma^k(S)\sigma^k(v_i) P_2 \sigma^{2k}(y)\sigma^{2k}(Q)\sigma^{2k}(z),
\]
but this path cannot be shorter than the distance between its endpoints, yielding a contradiction.
\end{proof}

We call a path $P = (p_i)_{i \in I}$ \emph{$\mu$-almost periodic} if there exists a pair of constants $m,l \in \mathbb N$ with $l < \mu$ such that $\sigma^{m}(p_i) = p_{i+l}$ whenever $\inf(I) + \mu < i < \sup(I)-\mu-l$.
Note that a ray is $\mu$-almost periodic if and only if every initial segment is $\mu$-almost periodic, and a double ray is $\mu$-almost periodic if and only if it is periodic. We call $l$ a period of the almost periodic path, and we say that the segment of $P$ between $p_i$ and $p_{i+l}$ for $\mu < i < N-\mu-l$ is a periodic piece.

We  also need the following notion of most counter-clockwise geodesics. Let $x$ and $y$ be vertices of a graph $G$ embedded in an orientable surface, and let $e_0$ be an edge incident to $x$ which lies on a geodesic from $x$ to $y$. The \emph{most counter-clockwise} geodesic from $x$ to $y$ starting with $e_0$ is defined by inductively taking $e_i$ such that $e_0,\dots,e_i$ is an initial piece of a geodesic from $x$ to $y$, and among all candidates picking the edge which maximises the (counter-clockwise) angle between $e_{i-1}$ and $e_i$ at their shared vertex. We can define \emph{most counter-clockwise} rays from a vertex tending to the right in the same way.

\begin{lemma}
\label{lem:almostperiodic}
Let $G$ be a graph embedded in the infinite cylinder $\mathbb R \times \mathbb R/\mathbb Z$ or the infinite strip $\mathbb R \times [0,1]$, and assume that the shift $\sigma$ induces an automorphism of $G$. Let $P$ be a most counter-clockwise geodesic path or ray in $G$.
\begin{enumerate}
    \item If a bypath uses an edge $e$ incident to $v\in P$ and $f_1$ and $f_2$ are the edges before and after $v$ on $P$, respectively, then $e$ lies between $f_1$ and $f_2$ in counter-clockwise direction.
    \item There is some $\mu$ which only depends on the embedding such that $P$ is $\mu$-almost periodic.
\end{enumerate}
\end{lemma}

\begin{proof}
    The first statement immediately follows from the observation that if there was a bypath attached on the ``counter-clockwise'' side, then we would have chosen the first edge of this bypath instead of the edge of $P$ in the corresponding step of the construction.

    For the second statement, let us call two vertices equivalent, if they lie in the same orbit with respect to the shift $\sigma$. A path is called \emph{nice}, if it starts and ends at equivalent vertices, but does not contain any other pair of equivalent vertices. A path $Q$ is called a \emph{copy} of another path $Q'$, if there is some $k$ such that $Q' = \sigma^k(Q)$.
    
    We first show that if $P$ contains multiple copies of a nice path $Q$, then these copies are necessarily consecutive. To this end, let $Q_1, \dots, Q_r$ be the set of copies of $Q$ contained in $P$. Note that these copies are necessarily disjoint because all their endvertices are equivalent, and there are no interior vertices of $Q$ equivalent to these endvertices. We claim that the segment $P_i$ of $P$ between $Q_i$ and $Q_{i+1}$ is trivial. 
    
    Assume not. Let $k$ be such that $\sigma^k$ maps the first vertex of $Q_i$ to the last vertex of $Q_{i}$ and let $l$ be such that $\sigma^l$ maps the first vertex of $P_i$ to the last vertex of $P_{i}$. Let $e$ be the first edge of $Q_i$ and let $e'$ be the first edge of $P_i$. If we replace $Q_i$ and $P_i$ in $P$ by $\sigma^{-k}(P_i)$ and $\sigma^l(Q_i)$, we obtain another geodesic. Since $P$ was most counter-clockwise, this means that $e$ cannot come after $\sigma^{-k}(e')$ in counter-clockwise order. A similar argument using $P_i$ and $Q_{i+1}$ shows that $e'$ cannot come after $\sigma^{k}(e)$ in counter-clockwise order. But this means that the first edge of $P_i$ must be a copy of the first edge of $Q_i$. Continuing inductively (and recalling that $Q_i$ does not contain a pair of equivalent vertices) we conclude that $P_i$ contains a copy of $Q$ as an initial segment, contradicting the assumption that $Q_1, \dots, Q_r$ was a complete set of copies of $Q$ contained in $P$.

    Next we define a decomposition of $P$ into (finite) subpaths. Let $v_0$ be the initial vertex of $P$ and inductively define $v_i$ as follows. If there is a nice subpath of $P$ starting at $v_{i-1}$, let $v_i$ be the last vertex of this subpath. Otherwise, let $v_i$ be the first vertex after $v_{i-1}$ for which there is a nice subpath starting at $v_i$, and let $v_i$ be the final vertex of $P$ if no such vertex exists. Let $P_i$ be the subpath of $P$ connecting $v_{i-1}$ to $v_i$. If any interior vertex of $P_i$ was equivalent to any other vertex of $P_i$, then this vertex would have been chosen as one of the $v_i$ according to our construction. Hence every $P_i$ is either nice, or does not contain any pair of equivalent vertices. 

    We now focus on the nice pieces of this decomposition.
    Define the reach of a nice path $Q$ as the $k$ such that $Q$ connects $x$ to $\sigma^k(x)$ and let $M$ be the maximal reach of a nice path. Note that the maximum exists because there are only finitely many nice paths up to taking copies. We claim that if there are two nice paths $Q$ and $Q'$ such that at least $M$ of the pieces $P_i$ are copies of $Q$ and $Q'$, respectively, then $Q$ is a copy of $Q'$.

    Assume to the contrary that $Q$ and $Q'$ are not copies of one another and let $r$ and $r'$ be the reaches of $Q$ and $Q'$, respectively. Since both $r$ and $r'$ are at most $M$, we can (without loss of generality) assume that $P$ contains a subpath $R$ obtained as the concatenation of $\sigma^{ir}(Q)$ for $0 \leq i < r'$ and a subpath $R'$ obtained as the concatenation of $\sigma^{ir'}(Q)$ for $0 \leq i < r$. We may further assume that $R$ and $R'$ are disjoint, and that the subpath $S$ of $P$ connecting $R$ to $R'$ does not contain any copies of $Q$ or $Q'$. 
    
    Note that $R$ connects some vertex $x$ to $\sigma^{rr'}(x)$, and $R'$ connects some vertex $y$ to $\sigma^{rr'}(y)$. Moreover, $S$ connects $\sigma^{rr'}(x)$ to $y$. We can thus define paths 
    \[B = \sigma^{rr'}(R)\sigma^{2rr'}(x)\sigma^{rr'}(S) \qquad \text{and} \qquad B' = \sigma^{-rr'}(S)\sigma^{-rr'}(y)\sigma^{-rr'}(R').\] 
    We claim that both of these are bypaths of $P$. Indeed, if $R$ was strictly shorter than $R'$, then $B$ is shorter than $\sigma^{rr'}(x) P \sigma^{rr'}(y)$, and if $R'$ was strictly shorter than $R$, then $B'$ would be shorter than $x P y$. It follows that $R$ and $R'$ have the same length, and thus $B$ has the same length as $\sigma^{rr'}(x) P \sigma^{rr'}(y) = SyR'$, and $B'$ has the same length as $xPy = R \sigma^{rr'}(x)S$.

    Considering the first edge of $\sigma^{-rr'}(S)$ not contained in $R$ and the first edge of $\sigma^{rr'}(R)$ not contained in $S$ we see that $B$ and $B'$ attach to different sides of $P$ which contradicts the first part of the lemma. We have thus shown that (up to taking copies) there is at most one nice path $Q$ for which at least $M$ of the pieces $P_i$ are copies of $Q$.

    Let $N$ be the number of vertices contained in each cell of the embedding. If we divide $P$ into subpaths $P_j'$ of equal length $2N$ and possibly a shorter final subpath, then each $P_j'$ except the last one contains two copies of the same vertex and thus contains a nice path. Moreover, if more than one copy of some nice $P_i$ is contained in $P$, then a copy of $P_i$ is contained in at least one $P_j'$ (since every copy of $P_i$ has length at most $N$). If the length of $P$ is at least $M N^{2N}$, then $M$ of the $P_j'$ must be pairwise copies of one another and thus all contain copies of the same nice $P_i$. As observed above, the copies of $P_i$ contained in $P$ must be consecutive.

    If the part of $P$ before the first $P_j'$ containing a copy of $P_i$ had length at least $M N^{2N}$, then we could apply the same argument to find another nice path for which at least $M$ copies are contained in $P$ which cannot happen. An analogous argument applies for the part of $P$ after the last $P_j'$ containing a copy of $P_i$. Thus $P$ is $\mu$-almost periodic for $\mu = M N^{2N} + 2N$.
\end{proof}

\begin{lemma}
    \label{lem:topmost}
    Assume that $G$ is embedded on an infinite strip $\mathbb R \times [0,1]$. Every band has a topmost geodesic, that is, a geodesic for which there is no bypath above it.
\end{lemma}

\begin{proof}
    Apply Lemma \ref{lem:almostperiodic} to the graph induced by the vertices in the band $\mathbb B$ to obtain an almost periodic geodesic ray $P_0$ tending to the right. By removing an initial piece, we may assume that $\sigma^k(P_0) \subseteq P_0$. Now define $P_i = \sigma^{-ik}(P_0)$, and let $P$ be the union of all $P_i$. Clearly, $P = \sigma^k(P)$, so $P$ is periodic. Moreover, if $x$ and $y$ are two vertices of $P$, then there is some $i$ such that $\sigma^{ik}(x)$ and $\sigma^{ik}(y)$ are vertices of $P_0$, and since the segment between $\sigma^{ik}(x)$ and $\sigma^{ik}(y)$ on $P_0$ is a geodesic, so is the segment between $x$ and $y$ on $P$. This shows that $P$ is a geodesic. Finally, there is no bypath on top of $P$ because no bypath attaches to the top of $P_0$ by Lemma~\ref{lem:almostperiodic}.
\end{proof}

The final lemma in this section is a key ingredient in our strategy. Essentially, it states that if two geodesics are separated by the topmost geodesic of a band, then their bypaths cannot intersect.

\begin{lemma}
    \label{lem:nodoublebypath}
    Assume that $G$ is embedded on an infinite strip $\mathbb R \times [0,1]$. Let $Q$ be the topmost geodesic of a band. Let $P_1$ be a geodesic embedded completely above $Q$, and assume that there is a vertex $q$ of $Q$ such that prepending $q$ to $P_1$ gives a geodesic. Let $P_2$ be a geodesic which does not contain any vertex above $Q$. Then no vertex of $G$ is simultaneously contained in a bypath of $P_1$ and $P_2$.
\end{lemma}
\begin{proof}
If a bypath of $P_1$ contained a vertex above $Q$, then it would contain a geodesic between two vertices of $Q$ which lies above $Q$. This would constitute a bypath of $Q$ contradicting the assumption that $Q$ is the top geodesic in its band.

Thus, if there was a vertex simultaneously contained in a bypath of $P_1$ and $P_2$, then there must be a vertex of $Q$ which lies on a bypath of $P_2$. Assume for a contradiction that there is such a vertex $x$. Let $B$ be the bypath, and let $u$ and $v$ be the endpoints of $B$. Let $B' = qP_1uBx$. Since $P_1 + q$ is a geodesic, and $B$ is a bypath of this geodesic, if we extend $B'$ by following $B$ to $v$, then we obtain a geodesic from $q$ to $v$. Thus $B'$ is a geodesic from $q$ to $x$, implying that $B'$ contains a bypath of $P_2$. Thus the first vertex of $P_1$ lies on a bypath of $Q$, giving the desired contradiction.
\end{proof}

\section{Wide Shadows and Guarding Strategies}
\label{subsec:wideshadow}
Let $G$ be a graph and let $H$ be a subgraph of $G$. We say that a cop $c_i$ \emph{guards} $H$, if $c_i$ is at a vertex of $H$ and plays a strategy which ensures that if the robber ever enters $H$, then $c_i$ captures the robber.

A particularly useful way of guarding if $H$ is a geodesic path in $G$ is staying within the \emph{wide shadow} of the robber, a notion introduced in~\cite{GonzalezMohar2024} which is defined as follows. Let $G$ be a graph, and let $P$ be a (finite or infinite) geodesic. Let $v$ be a vertex of $G$. For a vertex $w$ on $P$, define 
\[\gamma(P,w,v) = \{u \in P \mid d(u,w) \leq d(v,w)\},\]
that is, $\gamma(P,w,v)$ is the intersection of the path $P$ with the ball of radius $d(v,w)$ centred at $w$.
The \emph{wide shadow} of $v$ on $P$, denoted by $S_P (v)$, is defined as
\[
    S_P (v) = \bigcap_{w \in P} \gamma(P,w,v).
\]
In other words, a vertex $u$ of $P$ is contained in the wide shadow of $v$ if  $d(u,w) \leq d(v,w)$ for every vertex $w$ of $P$. Note that if $v \in P$, then the wide shadow only consists of $v$. It is also not hard to see that the wide shadow is a subpath of $P$ (since it is defined as the intersection of a set of paths). Moreover, we have the following result whose proof is analogous to the proof of Lemmas 2.3 and 2.8 in~\cite{GonzalezMohar2024}.

\begin{lemma}
    \label{lem:wideshadow}
    Let $G$ be a graph, let $P$ be a (finite or infinite) geodesic in $G$, and let $v \in V(G)$. 
    \begin{enumerate}
        \item The wide shadow of $v$ on $P$ is non-empty.
        \item The wide shadow of $v$ on $P$ has size $1$ if and only if $v$ lies on a (possibly trivial) bypath of $P$.
        \item If $v'$ is a neighbour of $v$ and $u$ is in the wide shadow of $v$ on $P$, then there is a neighbour of $u$ which is contained in the wide shadow of $v'$ on $P$.
    \end{enumerate} 
\end{lemma}

Note that by the above lemma, once a cop is in the wide shadow of the robber's position on some path $P$, this cop can guard $P$ by remaining in the wide shadow indefinitely. Further note that if a vertex $u$ is in the wide shadow of the robber with respect to two different geodesics $P$ and $Q$ and a cop guarding $P$ is positioned at $u$, then this cop can immediately switch to guarding $Q$ instead. We will repeatedly make use of this observation in the proof of our main result.

We now give an equivalent definition of the wide shadow which is sometimes easier to work with than the definition given above.

\begin{lemma}\label{lem:wideshadowextra}
    Let $P$ be a geodesic in a graph $G$ and let $v \in V(G)$. Then there are two vertices $w_1,w_2$ in $P$ such that
    \[
    S_P(v) = \gamma(P,w_1,v) \cap \gamma(P,w_2,v).
    \]
    
    Moreover, there are vertices $w_1'$ and $w_2'$ such that $w_1$ can be chosen as any vertex which comes before $w_1'$ on $P$, and $w_2$ can be chosen as any vertex which comes after $w_2'$ on $P$.
\end{lemma}

\begin{proof}
Let $P = (v_i)_{i \in I}$. Note that $\gamma(P,v_i,v)$ is finite because it is a subset of a finite ball. Since any intersection of finite sets can be written as an intersection of finitely many of these finite sets, we can choose $a$ and $b$ such that \[S_P (v) = \bigcap_{a \leq i \leq b} \gamma(P,v_i,v).\] 
Decreasing $a$ or increasing $b$ gives the same intersection, and thus the arguments in the remainder of this proof apply to any small enough $a$ and large enough $b$, and we may without loss of generality assume that no vertex of $S_P(v)$ precedes $v_a$ on $P$, and no vertex of $S_P(v)$ succeeds $v_b$ on $P$.

Let $i < j$. We claim that if $v_i \in \gamma(P,v_j,v)$, then $v_i \in \gamma(P,v_{j-1},v)$. Assume not. Then $d(v_{j-1},v_i)>d(v_{j-1},v)$, and thus 
\[
 d(v_j,v) \geq d(v_j,v_i) = d(v_{j-1},v_i) + 1 > d(v_{j-1},v) +1,
\]
contradicting the fact that there is a path of length $d(v_{j-1},v) +1$ from $v_j$ to $v$ via $v_{j-1}$. By a symmetric argument, if $i > j$ and $v_i \in \gamma(P,v_j,v)$, then $v_i \in \gamma(P,v_{j+1},v)$.

This shows that if $a \leq i \leq b$, then $v_i$ is contained in $S_P(v)$ if and only if it is contained in $\gamma(P,v_a,v) \cap \gamma(P,v_b,v)$, and since $S_P(v) \subseteq \{v_i\mid a \leq i \leq b\}$ it follows that $S_P(v)=\gamma(P,v_a,v) \cap \gamma(P,v_b,v)$.
\end{proof}

Note that the above lemma can be used to characterise the boundary of $S_P(v)$ as follows.

\begin{corollary}
    \label{cor:wideshadowboundary}
    Let $P$ be a (finite or infinite) geodesic and let $v \in G$.
    \begin{enumerate}
        \item Assume that $P$ has a first vertex $w$. The last vertex in $S_P(v)$ is the (unique) vertex $u$ of $P$ satisfying $d(u,w) = \min(d(v,w), |P|)$.
        \item Assume that $P$ does not have a first vertex, and let $w_i$ be a sequence of vertices of $P$ such that $w_{i+1}$ comes before $w_i$ on $P$. The last vertex in $S_P(v)$ is the (unique) vertex $u$ of $P$ satisfying $d(u,w_i) = d(v,w_i)$ for all but finitely many $i$.
    \end{enumerate}
    Analogous statements hold mutatis mutandis for the first vertex in $S_P(v)$.
\end{corollary}

\begin{proof}
    By Lemma \ref{lem:wideshadow}, the wide shadow is non-empty, and by Lemma \ref{lem:wideshadowextra} the right boundary of the wide shadow is determined by any vertex that comes early enough on~$P$.
\end{proof}

The following lemma will be useful in allowing a cop to move while two other cops are guarding different paths, respectively.

\begin{lemma}
    \label{lem:leisurelyguarding}
    Assume that $c_1$ and $c_2$ are in the robber's wide shadow on geodesics $P_1$ and $P_2$, respectively. Let $q$ be an arbitrary vertex. Then there is a strategy for $c_1$ and $c_2$ with the following properties.

    \begin{enumerate}
        \item $c_1$ and $c_2$ stay in the robber's wide shadow on $P_1$ and $P_2$, respectively.
        \item For every $D \in \mathbb N$ there is a constant $N$ such that if there are at least $N$ consecutive steps for which $d(r,q) \leq D$ and $r$ does not visit any vertex which simultaneously lie on a bypath of $P_1$ and $P_2$, then on at least one of these $N$ steps at least one of $c_1$ and $c_2$ does not move.
    \end{enumerate}

    Moreover, if $c_2$ starts out in the first or second vertex of the wide shadow (with respect to the order provided by $P_2$), then $c_2$ stays in the first or second vertex of the wide shadow throughout. In this case, if $c_2$ ever switches from the first to the second vertex, $c_2$ does not move on the respective turn.
\end{lemma}

\begin{proof}
    Fix $D$.
    Denote by $A$ and $B$ the sets of vertices at distance $\leq D$ from $q$ which do not lie on a bypath of $P_1$ and $P_2$, respectively. Note that by Lemma \ref{lem:wideshadow}, $c_1$ and $c_2$ can stay in the robber's wide shadow, so all we need to show is that one of the two can eventually stay at the same vertex if the robber stays in $A \cup B$.

    Since $A \cup B$ is finite, Lemma \ref{lem:wideshadowextra} implies that there are vertices $w_{i,1}$ and $w_{i,2}$ on $P_i$ such that the wide shadow with respect to $P_i$ of each vertex $v \in A \cup B$ is $\gamma(P_i, w_{i,1}, v) \cap \gamma(P_i, w_{i,2}, v)$.

    If $r \in A$, then either $d(r,w_{1,1}) > d(c_1,w_{1,1})$, or $d(r,w_{1,2}) > d(c_1,w_{1,2})$. If both inequalities are satisfied, then $c_1$ can remain at its current position regardless of the robber's move. If (without loss of generality) $d(r,w_{1,1}) = d(c_1,w_{1,1})$, then the only way the robber can force $c_1$ to move is by decreasing the distance to $w_{1,1}$ in which case the equality $d(r,w_{1,1}) = d(c_1,w_{1,1})$ remains true after the cops' corresponding move. Since the robber is caught if $d(r,w_{1,1}) = d(c_1,w_{1,1}) = 0$, this can only happen finitely many times, thus showing that if the robber stays in $A$ indefinitely, then $c_1$ can eventually stay at the current vertex for one move. An analogous argument shows that if the robber stays in $B$, then $c_2$ can remain at its current vertex for a move eventually.

    It only remains to investigate the case where the robber moves in and out of $A$ an unbounded number of times. Consider the first time the robber moves from $B \setminus A$ to $A$. Before this move, the robber was on a bypath of $P_1$, so  $d(r,w_{1,1}) = d(c_1,w_{1,1})$ and $d(r,w_{1,2}) = d(c_1,w_{1,2})$ must have been satisfied. If the robber's position $r'$ after this move satisfied $d(r',w_{1,1}) < d(c_1,w_{1,1})$, then $d(r',w_{1,1}) = d(c_1,w_{1,1}) - 1$ and $d(r',w_{1,2}) \leq d(r',r) + d(r,w_{1,2}) = 1+ d(r,w_{1,2})$, hence the robber is still on a bypath of $P_1$ contradicting the assumption that $r' \in A$. Thus $d(r',w_{1,1}) \geq  d(c_1,w_{1,1})$ and (by an analogous argument) $d(r',w_{1,2}) \geq d(c_1,w_{1,2})$, showing that $c_1$ does not have to move in order to remain in the wide shadow of the robber.

    For the `moreover' part note that the strategies of $c_1$ and $c_2$ are independent of what the other cop is doing. Hence, $c_2$ can (without affecting $c_1$) pretend to play on the graph $G'$ where each vertex $x$ of $B$ is connected to the vertex of $P_2$ at the same distance as $x$ from $w_{2,2}$. This does not change the fact that $x$ does not lie on a bypath of $P_2$, but it reduces the wide shadow on $P_2$ to the first two vertices of the wide shadow in the original graph. If $c_2$ starts out at the first vertex of the wide shadow and moves on every turn, then $c_2$ can ensure to always be at the first vertex of the wide shadow.
\end{proof}

\section{Proof of the Main Result}
\label{sec:mainresult}

Throughout this section, we let $G$ be a cover of a graph $H$ embedded on the torus, projective plane, or Klein bottle as described in Lemma \ref{lem:cover}. We may assume that $G$ is embedded on an infinite cylinder surface $\mathbb R \times \mathbb R/\mathbb Z$ and that the shift $(x,y)\mapsto (x+1,y)$ induces an automorphism $\sigma$ of $G$. Our goal is to prove the following theorem from which our main result follows by Lemma \ref{lem:sigmateleportation}.

\begin{theorem}\label{thm:configurations}
    Let $G$ be a cover of a graph embedded on the torus, projective plane, or Klein bottle as described in Lemma \ref{lem:cover}. Then $3$ cops can win the $\sigma$-cops and robber game where only 2 cops are allowed to move at any turn. 
\end{theorem}

    The strategy, and thus also the proof of this theorem is split into multiple phases. At the end of each phase we have two periodic geodesics $P$ and $Q$ (where either $P=Q$, or $P$ and $Q$ are disjoint) and a vertex $q$ on $Q$. Split $P$ in the graph, and embed the resulting graph into a strip as in Remark \ref{rmk:stripembedding}. Once we have ensured that the robber cannot cross $P$, we may assume that the game takes place on this graph embedded in a strip. We denote the copy of $P$ at the top boundary of the strip by $P$. If $P=Q$ we denote the copy at the bottom boundary by $Q$, otherwise $Q$ is disjoint from $P$ and thus is a geodesic in the graph embedded in a strip.

    With this setup, after each phase, the following properties hold, see Figure \ref{fig:valid-configuration}.
    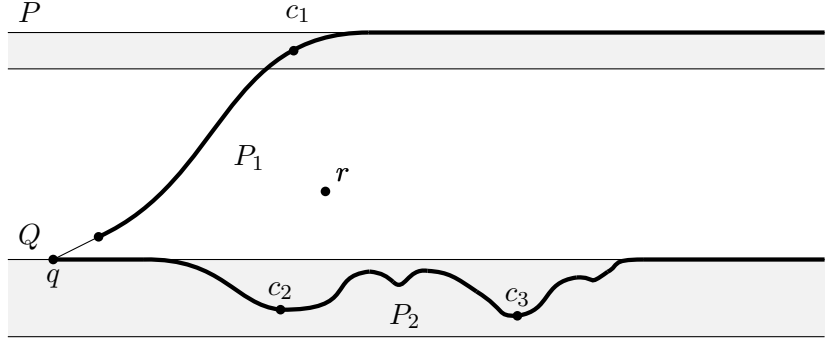
\begin{figure}
\centering
\begin{tikzpicture}[scale=0.6]
  \coordinate (A) at (-4.00, 8.00);
  \coordinate (P) at (14.00, 8.00);
  \coordinate (C) at (-4.00, 3.00);
  \coordinate (Q) at (14.00, 3.00);
  \coordinate (E) at (-4.00, 1.30);
  \coordinate (F) at (14.00, 1.30);
  \coordinate (q) at (-3.00, 3.00);
  \coordinate (G) at (3.95, 8.00);
  \coordinate (P_1) at (0.75, 4.75);
  \coordinate (I) at (0.75, 8.00);
  \coordinate (r) at (3.00, 4.50);
  \coordinate (c_1) at (1.00, 6.5);
  
  \coordinate (H) at (-1.00, 3.00);
  \coordinate (c_2) at (2.01, 1.895);
  
  \coordinate (K) at (0.75, 3.07);
  
  \coordinate (L) at (1.00, 2.00);
  \coordinate (M) at (3.97, 2.735);
  \coordinate (N) at (3.75, 1.875);
  \coordinate (O) at (3.00, 2.625);
  \coordinate (R) at (4.49, 2.510);
  \coordinate (S) at (4.26, 2.705);
  \coordinate (T) at (4.35, 2.625);
  \coordinate (U) at (5.25, 2.750);
  \coordinate (P_2) at (4.75, 2.250);
  \coordinate (W) at (4.80, 2.825);
  \coordinate (X) at (6.33, 2.295);
  \coordinate (Y) at (5.72, 2.725);
  \coordinate (Z) at (5.88, 2.600);
  \coordinate (AA) at (6.41, 2.255);
  \coordinate (c_3) at (7.23, 1.760);
  \coordinate (D) at (-2.00, 3.500);
  \coordinate (J) at (7.00, 1.875);
  \coordinate (V) at (6.75, 1.750);
  \coordinate (AB) at (8.54, 2.605);
  \coordinate (AC) at (8.00, 1.875);
  \coordinate (AD) at (7.75, 2.500);
  \coordinate (B) at (9.22, 2.730);
  \coordinate (AE) at (9.00, 2.625);
  \coordinate (AF) at (8.74, 2.385);
  \coordinate (AH) at (9.50, 2.875);
  
  \coordinate (AI) at (9.25, 3.00);
  \coordinate (AG) at (10.00, 3.00);

\draw[fill=lightgray, opacity = 0.2, draw = none] (-4,8) rectangle (14,7.2);
\draw[fill=lightgray, opacity = 0.2, draw = none] (-4,3) rectangle (14,1.3);
  \draw (A) -- (P);
  \draw (C) -- (Q);
  \draw (E) -- (F);
  \draw [ultra thick] (D) .. controls (0.75, 4.75) and (0.75, 8.00) .. (G);
\draw [ultra thick] (q) -- (H);
\draw [ultra thick] (H) .. controls (0.75, 3.035) and (1.00, 2.000) .. (c_2);
\draw [ultra thick] (c_2) .. controls (3.75, 1.875) and (3.00, 2.625) .. (M);
\draw [ultra thick] (M) .. controls (4.26, 2.705) and (4.35, 2.625) .. (R);
\draw [ultra thick] (R) .. controls (4.75, 2.250) and (4.80, 2.825) .. (U);
\draw [ultra thick] (U) .. controls (5.72, 2.725) and (5.88, 2.600) .. (X);
\draw (q) -- (D);
\draw [ultra thick] (X) .. controls (7.00, 1.875) and (6.75, 1.750) .. (c_3);
\draw [ultra thick](c_3) .. controls (8.00, 1.875) and (7.75, 2.500) .. (AB);
\draw [ultra thick] (AB) .. controls (9.00, 2.625) and (8.74, 2.385) .. (B);
\draw [ultra thick] (B) .. controls (9.50, 2.875) and (9.25, 3.000) .. (AG);
\draw [ultra thick] (AG) -- (Q);
\draw [ultra thick] (G) -- (P);
  \fill[black] (q) circle (2pt);
  \node[below, black] at (q) {$q$};
  \fill[black] (r) circle (2pt);
  \node[above right, black] at (r) {$r$};
  \fill[black] (c_2) circle (2pt);
  \fill[black] (c_3) circle (2pt);
  \node[above, black] at (c_3) {$c_3$};
  \fill[black] (D) circle (3pt);
  \fill[black] (2.3,7.6) circle (3pt);
  \node[above, black] at (2.4,8) {$c_1$};
  
  \node[above right, black] at (A) {$P$};
  \node[above right, black] at (C) {$Q$};
  \fill[black] (q) circle (3pt);
  \node[above right, black] at (P_1) {$P_1$};
  \fill[black] (r) circle (3pt);
  \node[above right, black] at (r) {$r$};
  \fill[black] (c_2) circle (3pt);
  \node[above, black] at (c_2) {$c_2$};
  \node[below, black] at (P_2) {$P_2$};
  \fill[black] (c_3) circle (3pt);

  \draw (-4, 7.2) -- (14,7.2);
\end{tikzpicture}
\caption{A valid configuration. The geodesic rays $P_1$ and $P_2$ are drawn as bold lines, the bands of $P$ and $Q$ are shaded in gray.}
\label{fig:valid-configuration}
\end{figure}

\begin{itemize}
    \item Cop $c_1$ guards a geodesic ray $P_1$ starting at a neighbour of $q$ outside the band of $Q$ and eventually following $P$ to the right such that $P_1 + q$ is still a geodesic ray.
    \item Cop $c_2$ guards a geodesic ray $P_2$ in the band of $Q$ starting at $q$ and eventually following $Q$ to the right.
    \item Cop $c_3$ is on $P_2$.
    \item The robber territory, that is, the set of all points which can be reached from the robber's current position $r$ without crossing $P_1$ or $P_2$, is contained in the area between $P_1$ and $P_2$ to the right of $q$.
    \item The robber territory contains only finitely many vertices in the band of $Q$; in some sense this means that $Q$ is the top-most geodesic in its band (if $P = Q$, then it means that after splitting, the copy denoted by $Q$ forms a band by itself). To see why this is possible even if there is only one band in the original graph, see Lemma~\ref{lem:gettingstarted} below.
\end{itemize}

We call a configuration as above a \emph{valid configuration}. We call a configuration \emph{weakly valid} if it satisfies all conditions except the third one, that is, there is no restriction on the position of $c_3$. Given a (weakly) valid configuration, let us call a band a \emph{robber band}, if it lies (strictly) between the band containing $P$ and the band containing $Q$.

Using these definitions, our strategy for proving Theorem \ref{thm:configurations} can be summarised as follows. We first give a strategy for the cops to build an initial valid configuration, see Lemma \ref{lem:gettingstarted}. Next, in Lemma \ref{lem:norobberbands} we show that if a valid configuration has no robber band, then the cops have a winning strategy. Finally, Lemma \ref{lem:reducebands} shows that there is a way for the cops to reduce the number of robber bands.

Observe that the definitions of $P_1$ and $P_2$ immediately imply that Lemma \ref{lem:nodoublebypath} applies, and thus there is no vertex simultaneously on a bypath of $P_1$ and $P_2$. This will be crucial as it allows us to apply Lemma \ref{lem:leisurelyguarding} and (provided that the robber's position remains bounded) move $c_3$ into position while $c_1$ and $c_2$ keep guarding these geodesics.

\begin{lemma}
    \label{lem:gettingstarted}
    There is a strategy using three cops where only two cops have to move at any given turn such that after finitely many turns we have a valid configuration.
\end{lemma}

\begin{proof}
    We start by defining a suitable double ray $P$, proceeding similarly to the proof of Lemma \ref{lem:topmost}. More precisely, let $v$ be any vertex and let $P_0$ be a most counter-clockwise geodesic ray starting at $v$ and tending to the right. By Lemma \ref{lem:almostperiodic}, $P_0$ is $\mu$-almost periodic for some $\mu$, and by removing an initial piece, we may assume that $\sigma^k(P_0)$ is a subgraph of $P_0$ for some $k$. Define $P_i = \sigma^{-ik}(P_0)$, and let $P$ be the union of all $P_i$. The exact same argument as in Lemma \ref{lem:topmost} shows that $P$ is a periodic geodesic, and that all bypaths of $P$ attach to the same side of $P$. In particular, if we apply the construction of Remark \ref{rmk:stripembedding} to $P$, then one of the two copies of $P$ has no bypath in the resulting graph.
    
    We now proceed with the cops' strategy. All cops start at the same vertex $v$ of $P$. Once the robber has chosen a starting vertex $r$, cop $c_1$ teleports to a vertex $\sigma^k(v)$ on $P$ which comes before the wide shadow of $r$ on $P$, and $c_2$ teleports to a vertex $\sigma^k(v)$ which comes after the wide shadow. In every subsequent turn, $c_1$ and $c_2$ move towards each other on $P$ until one of them (without loss of generality $c_2$) reaches the wide shadow of the robber. Note that this must happen since by Lemma \ref{lem:wideshadow}, every single robber step moves the boundaries of the wide shadow at most one step along $P$, and thus the wide shadow cannot skip over the cops' positions.

    From then on, $c_2$ remains in the wide shadow of the robber on $P$. Since the robber cannot cross $P$, we may split $P$ and henceforth assume that the game takes place in this split graph embedded in an infinite strip described in Remark \ref{rmk:stripembedding}. Denote the copy of $P$ at the top of this strip by $P$, and the copy at the bottom by $Q$. As noted above, one of the two copies has no bypath, and we may assume without loss of generality that this copy is $Q$.
    
    Let $S$ be a geodesic ray starting at $Q$ and eventually following $P$ to the right; such a ray exists by Lemma \ref{lem:ray-q-to-P}. We may without loss of generality assume that the second vertex of $S$ does not lie in the band of $Q$. Pick a shift $P_1 = \sigma^{-k} (S)$ so that the wide shadow of the robber's position on $P_1$ is entirely contained in $P$. Denote the first vertex of $P_1$ by $q$.
    
    By picking $k$ large enough, we may assume that there is a vertex in $P_1 \cap P$ coming before the wide shadow of $r$ on $P_1$ which $c_3$ can teleport to. Cop $c_3$ teleports to this vertex, and cop $c_1$ teleports to some vertex of $P$ which comes after the wide shadow of $r$ on $P_1$. In subsequent turns, $c_1$ moves towards $c_3$ on $P_1$ until either $c_1$ or $c_3$ is contained in the robber's wide shadow. Note that unlike before, $c_1$ and $c_3$ cannot both move because $c_2$ is guarding $P$ and therefore potentially has to move on every single turn. Once one of them is contained in the robber's shadow, that cop (without loss of generality $c_1$) starts guarding $P_1$, or more precisely, $P_1 - q$, and the other cop can stop moving.

    Letting $P_2$ be the sub-ray of $Q$ starting at $q$ and tending to the right we notice that $c_2$ guards $P_2$. Recalling that $P=Q$ in the original graph, we may assume that $c_3$ is already positioned on $Q$, and without loss of generality on $P_2$ and we have thus arrived at a valid configuration.
\end{proof}

\begin{lemma}
    \label{lem:norobberbands}
    If in a weakly valid configuration there is no robber band, then the cops have a strategy to guarantee capture after finitely many steps.
\end{lemma}

\begin{proof}
    Note that it is sufficient to show that we can reduce the game to the following situation from which we can follow the proof of~\cite[Theorem 4.4]{GonzalezMohar2024}:
    \begin{itemize}
        \item Cops $c_1$ and $c_2$ guard (finite) paths $P_1$ and $P_2$ starting and ending at the same vertices, respectively, 
        \item The robber territory $R$ is the (finite) subgraph embedded in the finite region bounded by $P_1$ and $P_2$,
        \item $P_1$ is a geodesic in the subgraph induced by $R$ and $P_1$,
        \item $P_2$ is a geodesic and has no bypath in the subgraph induced by $R$ and $P_2$.
    \end{itemize}

    We first show that the cops have a strategy to reduce the number of vertices in the robber territory that lie between $Q$ and $P_2$, in other words, we may eventually assume that $P_2$ is contained in $Q$.
    
    If $c_2$ is ever on $Q$, then (since both $q$ and any vertex far enough along $P_2$ are contained in $Q$) it follows from Lemma \ref{lem:wideshadowextra} that $c_2$ is in the wide shadow of the robber on the geodesic subray $qQ$ of $Q$ starting at $q$. Thus $c_2$ can switch to guarding $qQ$ without changing positions. 

    Hence we may assume that the distance between the robber and $q$ remains bounded. By Lemma \ref{lem:nodoublebypath}, no vertex is contained in a bypath of both $P_1$ and $P_2$. Thus, by Lemma \ref{lem:leisurelyguarding} we may assume that there are infinitely many turns on which at least one of $c_1$ and $c_2$ does not move. 
    
    Let $P_2'\neq P_2$ be a geodesic ray in the band of $Q$ such that no vertex of $P_2'$ is embedded below $P_2$ and there is no bypath of $P_2$ strictly between $P_2$ and $P_2'$. Whenever $c_1$ or $c_2$ does not move, $c_3$ moves towards $P_2'$, and then along $P_2'$ towards the wide shadow of the robber. Since we assume that $d(r,q)$ is bounded, $c_3$ eventually arrives at a vertex in the wide shadow of the robber. At this point, the robber is either between $P_1$ and $P_2'$ or between $P_2'$ and $P_2$. In the first case, by switching from $P_2$ to $P_2'$ we have reduced the number of vertices between $Q$ and the ray which is guarded. In the second case, if $c_2$ and $c_3$ keep guarding $P_2$ and $P_2'$, respectively, then the robber is restricted to a finite set of vertices which does not contain any bypath of either of the two geodesics, and thus the cops have a winning strategy.

    Now assume that $P_2$ is contained in $Q$. Since $Q$ has no bypath in the robber territory we only need to reduce the infinite geodesics $P_1$ and $P_2$ to finite paths. We show that we can achieve this without moving the cops.

    Let $H$ be the subgraph of $G$ embedded in the strip between $P$ and $Q$, including all vertices of $P$ and $Q$ and all edges of $P$, but none of the edges of $Q$. Without loss of generality assume that $H$ is invariant under $\sigma$; otherwise replace $\sigma$ by a suitable power of $\sigma$. 
    
    Let $S_i$ be a most counter-clockwise path in $H$ from $q$ to $\sigma^i(q)$. By Lemma \ref{lem:almostperiodic} there is some $\mu$ such that each $S_i$ is $\mu$-almost periodic. Note that since the period of each $S_i$ is at most $\mu$, there are (up to shifts by $\sigma$) only finitely many possible periodic pieces, and thus infinitely many $S_i$ share the same periodic piece. Denote this piece $S$, its starting vertex by $x$ and its last vertex by $y = \sigma^m(x)$.

    Our next goal is to show that the double ray obtained by repeated copies of $S$ is a geodesic. Let $x'$ be the vertex of $P$ closest to $x$.  Moreover, let $\delta = d(x,x')$, let $l = d(x,\sigma^m(x))$ be the length of $S$, and let $l' = d(x',\sigma^m(x'))$ be the length of the segment of $P$ between $x'$ and $\sigma^m(x')$. If $l < l'$, then
    \begin{align*}
        d(x', \sigma^{(2\delta+1)m}(x')) 
        &\leq d(x,x') + d(x, \sigma^{(2\delta+1)m}(x)) + d(\sigma^{(2\delta+1)m}(x), \sigma^{(2\delta+1)m}(x'))\\
        &\leq 2\delta + (2 \delta + 1) l \\
        &\leq 2\delta + (2 \delta + 1) (l'-1)\\
        &< (2 \delta + 1) l',    
    \end{align*}
    contradicting the fact that $P$ is a geodesic. If $l' < l$, we can apply the same argument to some $S_i$ containing at least $(2 \delta + 1)$ copies of $S$. Hence $l = l'$.

    Next, let $T$ be a double ray composed of copies of $S$, that is, $T$ is the union of $\sigma^{im}(S)$ for $i \in \mathbb Z$. If $T$ was not a geodesic in $G$, then there would be some $i$ such that $d(x,\sigma^{im}x) < i l$, and an analogous computation as above would show that $d(x', \sigma^{(2\delta+1)im}(x')) < (2 \delta + 1) il'$, contradicting the fact that $P$ is a geodesic.

    Hence $T$ is a periodic geodesic and thus must be contained in some band. Since there are no robber bands, $T$ is contained in the same band as $P$. Thus there is a bypath of $T$ containing a vertex in $P$. Since $P$ corresponds to the upper boundary of the strip, if $T \neq P$, this bypath would lie on the counter-clockwise side of $T$, and thus there would be some $S_i$ for which we can find a bypath on the counter-clockwise side, contradicting Lemma \ref{lem:almostperiodic}. Hence $T=P$.

    An analogous argument as above shows that if we take most counter-clockwise geodesic from the robber's position $r$ to $\sigma^i(q)$ for infinitely many $i$, then there are infinitely many $i$ for which the periodic piece is contained in $P$. Hence the most counter-clockwise geodesics from $q$ and $r$ to $\sigma^i(q)$ intersect in arbitrarily large subpaths of $P$ as $i$ tends to infinity.

    Now let $i$ be large enough such that the robber's position is in the finite region between $S_i$ and $Q$, and the most counter-clockwise geodesics from $q$ and $r$ to $\sigma^i(q)$ intersect in a vertex $v$ on $P$ which comes after the position of $c_1$ on $P_1$. Let $P_1'=qP_1vS_i\sigma^i(q)$. We note that the piece of $P_1$ between $q$ and $v$ has the same length as the corresponding piece of $S_i$ since both paths are geodesics. It follows that $P_1'$ has the same length as $S_i$ and thus is a geodesic between $q$ and $\sigma^i(q)$. 

    We claim that the position $c_1$ of $c_1$ is already in the wide shadow of $r$ on $P_1'$. To see this, note that $d(q,c_1) \leq d(q,r)$ because $c_1$ is in the wide shadow of the robber on $P_1$. Moreover, \[
    d(\sigma^i(q),c_1) = d(\sigma^i(q),v) + d(v,c_1) \leq  d(\sigma^i(q),v) + d(v,r) = d(\sigma^i(q),r),
    \]
    where the second step again follows from the fact that $c_1$ is in the wide shadow of the robber on $P_1$, and the two equalities are due to the fact that most counter-clockwise geodesics from $q$ and $r$ to $\sigma^i(q)$ both go through the vertex $v$. By the `moreover' part of Lemma \ref{lem:wideshadowextra}, $c_2$ is in the wide shadow of $r$ as claimed. Thus, we are exactly in the situation we set out to achieve, and we can follow the proof of~\cite[Theorem 4.4]{GonzalezMohar2024}.
\end{proof}

Finally, we provide a strategy for the cops to reduce the number of robber bands.

\begin{lemma}\label{lem:reducebands}
    Starting from any valid configuration with robber bands, the cops have a strategy to obtain a valid configuration with fewer robber bands. 
\end{lemma}

\begin{proof} 
Throughout this proof, let us define the level $\ell(v)$ of a vertex $v$ by $\ell(v) = d(v,q)$.

Let us denote by $Q'$ the top geodesic of the band just above the band containing $Q$; note that this band is necessarily a robber band. Our strategy splits into two phases. In the first phase, the cops' aim is to move $c_3$ to $Q'$ while $c_2$ still is at $P_2$. In the second phase $c_2$ moves to the band of $Q'$ to arrive at a valid configuration where $Q'$ replaces $Q$.

Note that we may assume that $c_2$ starts out at the same level as the robber. If not, teleport $c_3$ to a vertex further from $q$ than $r$ and change $P_2$ so that it passes through the new position of $c_3$. If $\ell(r) < \ell(c_3)$ for long enough, then by Lemmas \ref{lem:leisurelyguarding} and \ref{lem:nodoublebypath} there is some step on which one of $c_1$ and $c_2$ does not move, so we can decrease the level of $c_3$ while still only moving 2 cops in each turn. Hence we eventually end up with $\ell(r) = \ell(c_3)$ at which point $c_2$ and $c_3$ swap roles. Moreover, by repeating the first step of this argument, we may assume that $\ell(c_3) > \ell(r)$.

Let $P_3$ be a most counter-clockwise geodesic ray starting at the position of $c_3$ and tending to the right in the subgraph of $G$ induced by all vertices between (and including) the bands containing $Q$ and $Q'$. By Lemma \ref{lem:almostperiodic}, $P_3$ is $\mu$-almost periodic and thus eventually follows a periodic geodesic to the right. Since all bypaths of $P_3$ attach to the same side, this periodic geodesic must be the topmost geodesic in its band. If it wasn't $Q'$, then by Lemmas \ref{lem:ray-q-to-P} and \ref{lem:switchgeodesic} there would be a more counter-clockwise geodesic than $P_3$. Thus, $P_3$ eventually follows $Q'$ to the right.

Let $x$ be the first vertex of $P_3$ and let $y$ be the first vertex of $P_3$ contained in $Q$. Note that we can extend $qP_2\sigma^k(q)$ to a periodic geodesic whenever  $\sigma^k(q) \in P_2$. Hence $P_2xP_3$ is a geodesic by Lemma \ref{lem:switchgeodesic}. Thus the lengths of $qQy$ and $P_2xP_3y$ coincide, showing that $P_2xP_3yQ$ is a geodesic. Cop $c_2$ is positioned to the left of $x$, so $c_2$ can switch to guarding this geodesic by Lemma \ref{lem:wideshadowextra} and we may thus without loss of generality assume that $P_2$ is this geodesic; in particular, $P_2$ coincides with $P_3$ between $x$ and $q'$.

Let $v_i$ denote the (unique) vertex of $P_3$ at level $i$, and let  $P_3^i$ be the subray of $P_3$ starting at $v_i$, that is, the ray obtained from $P_3$ by removing all vertices at level $<i$. We point out that this is only defined if $P_3$ has a vertex at level $i$.

If $P_3^i$ is defined and $\ell(v) > i$, then there either is a path from $v$ to $P_1$ in the robber territory only using vertices whose level is at least $i$ which does not cross $P_3^i$, or there is such a path from $v$ to $P_2$, but not both. In the first case we say that $v$ lies above $P_3$, in the second we say that $v$ lies below $P_3$. 

Split the vertices of $R \setminus P_3^i$ into three regions, and define cop strategies for each region as follows; see Figure \ref{fig:first-band-reduction-phase} for a sketch of the regions.
\begin{itemize}
    \item Region $A$ consists of vertices with $\ell(v) \leq i$. When the robber is in $A$, cops $c_1$ and $c_2$ maintain positions in the wide shadow of the robber on $P_1$ and $P_2$, respectively. The level of $c_2$ is the same or one less than the level of the robber. Cop $c_3$ is at $v_i$ or $v_{i+1}$, and if $c_3$ is at $v_i$, then $c_2$ is at the same level as the robber.
    \item Region $B$ consists of vertices with $\ell(v) > i$ above $P_3$. While the robber is in $B$, cops $c_1$ and $c_3$ maintain positions in the wide shadow of the robber; $c_3$ only moves if its current position is not in the wide shadow of the robber. Cop $c_2$ is at the unique vertex of $P_2$ at level $i$.
    \item Region $C$ consists of vertices with $\ell(v) > i$ below $P_3$. If the robber ever enters $C$, cops $c_2$ and $c_3$ henceforth maintain positions in the wide shadow of the robber on $P_2$ and $P_3$, respectively. Note that this immediately puts us in a valid configuration with no robber bands (with $y$ playing the role of $q$ and $Q'$ playing the role of $P$), so the cops can win the game by Lemma \ref{lem:norobberbands} in this case.
\end{itemize}

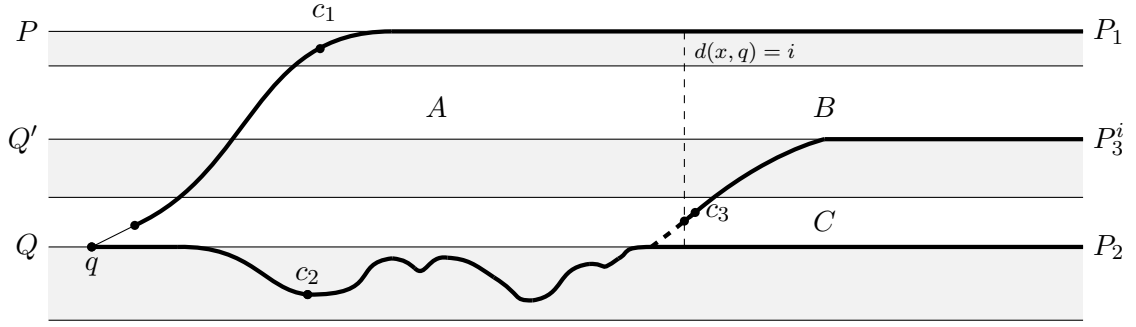
\begin{figure}
\centering
\begin{tikzpicture}[scale=0.57]
  \coordinate (A) at (-4.00, 8.00);
  \coordinate (P) at (11.00, 8.00);
  \coordinate (C) at (-4.00, 3.00);
  \coordinate (Q) at (11.00, 3.00);
  \coordinate (E) at (-4.00, 1.30);
  \coordinate (F) at (11.00, 1.30);
  \coordinate (q) at (-3.00, 3.00);
  \coordinate (G) at (3.95, 8.00);
  \coordinate (P_1) at (0.75, 4.75);
  \coordinate (I) at (0.75, 8.00);
  \coordinate (r) at (3.00, 4.50);
  \coordinate (c_1) at (1.00, 6.5);
  
  \coordinate (H) at (-1.00, 3.00);
  \coordinate (c_2) at (2.01, 1.895);
  
  \coordinate (K) at (0.75, 3.07);
  
  \coordinate (L) at (1.00, 2.00);
  \coordinate (M) at (3.97, 2.735);
  \coordinate (N) at (3.75, 1.875);
  \coordinate (O) at (3.00, 2.625);
  \coordinate (R) at (4.49, 2.510);
  \coordinate (S) at (4.26, 2.705);
  \coordinate (T) at (4.35, 2.625);
  \coordinate (U) at (5.25, 2.750);
  \coordinate (P_2) at (4.75, 2.250);
  \coordinate (W) at (4.80, 2.825);
  \coordinate (X) at (6.33, 2.295);
  \coordinate (Y) at (5.72, 2.725);
  \coordinate (Z) at (5.88, 2.600);
  \coordinate (AA) at (6.41, 2.255);
  \coordinate (c_3) at (7.23, 1.760);
  \coordinate (D) at (-2.00, 3.500);
  \coordinate (J) at (7.00, 1.875);
  \coordinate (V) at (6.75, 1.750);
  \coordinate (AB) at (8.54, 2.605);
  \coordinate (AC) at (8.00, 1.875);
  \coordinate (AD) at (7.75, 2.500);
  \coordinate (B) at (9.22, 2.730);
  \coordinate (AE) at (9.00, 2.625);
  \coordinate (AF) at (8.74, 2.385);
  \coordinate (AH) at (9.50, 2.875);
  
  \coordinate (AI) at (9.25, 3.00);
  \coordinate (AG) at (10.00, 3.00);


\draw[fill=lightgray, opacity = 0.2, draw = none] (-4,8) rectangle (20,7.2);
\draw[fill=lightgray, opacity = 0.2, draw = none] (-4,3) rectangle (20,1.3);
\draw[fill=lightgray, opacity = 0.2, draw = none] (-4,5.5) rectangle (20,4.15);

  \draw (-4, 4.15) -- (20,4.15);
  \draw (-4, 5.5) -- (20,5.5);
  \draw (-4, 7.2) -- (20,7.2);
  
  \draw (A) -- (20,8);
  \draw (C) -- (20,3);
  \draw (E) -- (20,1.3);
  \draw [ultra thick] (D) .. controls (0.75, 4.75) and (0.75, 8.00) .. (G);
  \draw[ultra thick] (G) -- (20, 8);
  \node[right] at (20,8) {$P_1$};
\draw [ultra thick] (q) -- (H);
\draw [ultra thick] (H) .. controls (0.75, 3.035) and (1.00, 2.000) .. (c_2);
\draw [ultra thick] (c_2) .. controls (3.75, 1.875) and (3.00, 2.625) .. (M);
\draw [ultra thick] (M) .. controls (4.26, 2.705) and (4.35, 2.625) .. (R);
\draw [ultra thick] (R) .. controls (4.75, 2.250) and (4.80, 2.825) .. (U);
\draw [ultra thick] (U) .. controls (5.72, 2.725) and (5.88, 2.600) .. (X);
\draw (q) -- (D);
\draw [ultra thick] (X) .. controls (7.00, 1.875) and (6.75, 1.750) .. (c_3);
\draw [ ultra thick] (c_3) .. controls (8.00, 1.875) and (7.75, 2.500) .. (AB);
\draw [ ultra thick] (AB) .. controls (9.00, 2.625) and (8.74, 2.385) .. (B);
\draw [ ultra thick] (B) .. controls (9.50, 2.875) and (9.25, 3.000) .. (AG);
\draw [dashed, ultra thick] (9.98, 3) .. controls (10.65, 3.5) and (11, 3.8) .. (11,3.8);
\fill[black] (11,3.8) circle (3pt);
\fill[black] (10.75,3.6) circle (3pt);
\draw[ultra thick] (10.75,3.6) -- (11,3.8);
\node [below, right] at (11,3.8) {$c_3$};
\draw [ultra thick] (11, 3.8) .. controls (12.65, 5.2) and (14, 5.5) .. (14,5.5);
\draw [ultra thick] (14, 5.5) -- (20, 5.5); 

\draw [dashed] (10.75,8) -- (10.75, 3);
\node [below,right, font=\scriptsize] at (10.75,7.5) {$d(x,q)=i$};

\node[right] at (20,5.5) {$P_3^i$};
\node[right] at (20, 3) {$P_2$};
\node at (5,6.25) {$A$};
\node at (14,6.25) {$B$};
\node at (14, 3.6) {$C$};

\draw [ultra thick] (10,3) -- (20,3);

  \fill[black] (q) circle (3pt);
  \node[below, black] at (q) {$q$};
  \fill[black] (c_2) circle (3pt);
  \fill[black] (D) circle (3pt);
  \fill[black] (2.3,7.6) circle (3pt);
  \node[above, black] at (2.4,8) {$c_1$};
  \node[left, black] at (A) {$P$};
  \node[left, black] at (C) {$Q$};
  \node[left, black] at (-4,5.5) {$Q'$};
  \fill[black] (q) circle (3pt);
  \fill[black] (c_2) circle (3pt);
  \node[above, black] at (c_2) {$c_2$};
\end{tikzpicture}
\caption{The first phase of the robber band-reduction strategy. Bold lines indicate the paths guarded by the three cops; when $c_3$ is on $P_2$ there is some overlap between $P_2$ and $P_3^i$. The dashed line at level $i$ indicates where the robber territory is split into regions $A$, $B$, and $C$.}
\label{fig:first-band-reduction-phase}
\end{figure}

We claim that the cops can follow this strategy in a way such that eventually one of three things happens. Either the robber is caught, or $c_3$ reaches $Q'$, or we may increase $i$. Note that increasing $i$ often enough also leads to $c_3$ reaching $Q'$. Clearly, once the cops have established positions as described in the strategies, the cops can play the strategy with only two cops moving at any given turn as long as the robber stays in the same region.

By assumption, the robber starts in region $A$, and the cops are in the correct positions for their strategy with $i = \ell(c_3)$. 
If the robber remains in $A$ long enough, then by Lemmas \ref{lem:leisurelyguarding} and \ref{lem:nodoublebypath}, there is some turn on which either $c_1$ or $c_2$ does not move. Each time this happens, $c_3$ moves one level up, and when $c_3$ is at level $i+2$ we may increase $i$ by one. Hence we may assume that the robber eventually moves to $B$ or $C$. 

When the robber moves from $A$ to $B$ or $C$, he moves from level $i$ to level $i+1$. If $c_3$ is at level $i+1$ then $c_3$ is already in the wide shadow of $r$ and does not have to move. Since $\ell(c_2)\in \{i-1,i\}$, cop $c_2$ can move to level $i$ in this turn. If $c_3$ is at level $i$, then $c_2$ is also at level $i$, so $c_2$ does not have to move and thus $c_3$ can move to level $i+1$. In either case, the cops are in the correct positions for their strategies for $B$ or $C$, respectively. 

As noted above, the cops may immediately switch to a winning strategy if the robber ever moves to $C$. Hence it only remains to consider the case when the robber is in $B$. If the robber moves far enough away from $q$, then the robber's wide shadow on $P_3^i$ is entirely contained in $Q'$ and thus $c_3$ must arrive at $Q'$. Otherwise, $\ell(r)$ remains bounded while the robber is in $B$. Since $P_3$ is a most counter-clockwise geodesic, by Lemma \ref{lem:almostperiodic} there is no bypath of $P_3$ in $B$. In particular, no vertex of $B$ simultaneously lies on a bypath of $P_1$ and $P_3^i$, and thus Lemma \ref{lem:leisurelyguarding} implies that there is some turn where either $c_1$ or $c_3$ does not move. On this turn, $c_2$ moves one level up. Since staying in the wide shadow never requires $c_3$ to move to $v_i$ (except for catching the robber), at this point we may increase $i$ by 1.

Finally, if the robber moves from $B$ to $A$, this means that the robber moves from level $i+1$ to level $i$. Hence $c_2$ is already at the same level as the robber. Thus $c_3$ may move to a vertex one level higher than $\ell(c_3)$ and we may increase $i$ to the level on which $c_3$ is after this move. Note that this implies that $i$ increases.

This shows that $c_3$ eventually reaches $Q'$ regardless of the robber's strategy, and we have thus completed the first phase of our strategy.

Before giving the strategy for the second phase, we show that we may assume that $\ell(c_3) = \ell(r)$. 

To prove this claim, first note that if $c_2$ is not in the wide shadow of the robber on $P_2$, then $c_3$ must have reached $Q'$ while guarding $P_3^i$, and the robber's position is necessarily above $Q'$. Before the move taking $c_3$ to $Q'$ there was a vertex in the wide shadow on $P_3^i$ which is not contained in $Q'$, and after the move there is no such vertex. This implies that $c_3$ is at the first vertex of the wide shadow (both before and after the move), and thus by Lemma \ref{lem:wideshadowextra} we have that $d(c_3,v_j) = d(r,v_j)$ for every large enough $j$. It follows from Lemma \ref{lem:wideshadowextra} and the fact that the wide shadow on $Q'$ is non-empty that $c_3$ is at the first vertex of the robber's wide shadow on $Q'$. Now $c_1$ and $c_3$ may guard $P_1$ and $Q'$ (by staying in the robber's wide shadow), respectively, while $c_2$ teleports to a vertex $c_2$ on a bypath of $P_2$ for which $d(q,c_2) > d(q,r)$. We modify $P_2$ so that it passes through this vertex $c_2$. If it ever happens that $\ell(r) = \ell(c_2)$, then $c_2$ is in the wide shadow of the robber. Otherwise, note that there is no bypath of $Q'$ above $Q'$, so by Lemmas \ref{lem:leisurelyguarding} there is a turn on which one of $c_1$ or $c_3$ does not move, and thus $c_2$ can move towards $q$. After finitely many such moves, $c_2$ is in the wide shadow of the robber on $P_2$.

Once $c_1$ and $c_2$ are in the wide shadow of the robber on $P_1$ and $P_2$, respectively, then $c_3$ can teleport to a vertex with $\ell(c_3)  > \ell(r)$, and the same backtracking argument as before ensures that eventually $\ell(c_3) = \ell(r)$. If the robber is between $Q'$ and $P$, then $c_2$ teleports to a higher level than the robber and we modify $P_2$ so that it passes through this position.

As in the first phase, our next step is to split the robber territory into regions and assign cop strategies to each region. To define the regions, we first define geodesics which the three cops will guard, respectively. 

Let $x$ and $y$ be the first and last vertex $P_1$ uses in the band of $Q'$, and note that $y$ necessarily lies on $Q$. Let $S$ be a geodesic ray in the band of $Q'$ starting at $y$, passing through $c_3$, and eventually following the lowest geodesic in this band to the right. Choose $k$ large enough that $\sigma^k(q)$ lies on $P_2$, $\sigma^k(x)$ lies on $S$, and $d(q,\sigma^k(q)) > 3m$ where $m$ is the largest distance of any vertex to $Q$. Let $P_1'$ be the subray obtained from $P_1$ by removing the initial piece up to $y$, let $P_2' = P_2\sigma^k(q)\sigma^k(P_1)\sigma^k(x)S$, and let $P_3' = P_1yS$. Note that $P_2'$ and $P_3'$ are geodesics by Lemma \ref{lem:switchgeodesic}. Finally, define $P_3'^i$ to be the subpath consisting of all vertices of $P_3'$ at level $\leq i$.

Using these geodesics, we define regions of the robber territory and corresponding strategies as follows, see Figure \ref{fig:second-band-reduction-phase} for a sketch.
\begin{itemize}
    \item Region $A$ consists of all vertices at level $\leq i$ which lie above $P_3'$. If the robber is in $A$, cops $c_1$ and $c_3$ maintain positions in the wide shadow of the robber on $P_1'$ and $P_3'^i$, respectively such that $\ell(c_3) \in \{\ell(r),\ell(r)-1\}$. Cop $c_2$ is at the unique vertex of $P_2'$ at level $i$.
    \item Region $B$ consists of all vertices at level $> i$. If the robber is in $B$, cops $c_1$ and $c_2$ maintain positions in the wide shadow of the robber on $P_1'$ and $P_2'$, respectively such that $\ell(c_2) \in \{\ell(r),\ell(r)-1\}$. Cop $c_3$ is at the vertex of $P_3$ at level $i$ or level $i-1$; if $\ell(c_3) = i-1$, then $\ell(c_2) = \ell(r)$.
    \item Region $C$ consists of all vertices at level $\leq i$ which lie below $P_3'$. If the robber is in $C$, cops $c_2$ and $c_3$ maintain positions in the wide shadow on $P_2'$ and $P_3'$, respectively. We note that this restricts the robber to a finite region enclosed by $P_2'$ and $P_3'$.
\end{itemize}

\begin{figure}
\centering

\begin{tikzpicture}[scale=0.57]
  \coordinate (A) at (-4.00, 8.00);
  \coordinate (P) at (11.00, 8.00);
  \coordinate (C) at (-4.00, 3.00);
  \coordinate (Q) at (11.00, 3.00);
  \coordinate (E) at (-4.00, 1.30);
  \coordinate (F) at (11.00, 1.30);
  \coordinate (q) at (-3.00, 3.00);
  \coordinate (G) at (3.95, 8.00);
  \coordinate (P_1) at (0.75, 4.75);
  \coordinate (I) at (0.75, 8.00);
  \coordinate (r) at (3.00, 4.50);
  \coordinate (c_1) at (1.00, 6.5);
  
  \coordinate (H) at (-1.00, 3.00);
  \coordinate (c_2) at (2.01, 1.895);
  
  \coordinate (K) at (0.75, 3.07);
  
  \coordinate (L) at (1.00, 2.00);
  \coordinate (M) at (3.97, 2.735);
  \coordinate (N) at (3.75, 1.875);
  \coordinate (O) at (3.00, 2.625);
  \coordinate (R) at (4.49, 2.510);
  \coordinate (S) at (4.26, 2.705);
  \coordinate (T) at (4.35, 2.625);
  \coordinate (U) at (5.25, 2.750);
  \coordinate (P_2) at (4.75, 2.250);
  \coordinate (W) at (4.80, 2.825);
  \coordinate (X) at (6.33, 2.295);
  \coordinate (Y) at (5.72, 2.725);
  \coordinate (Z) at (5.88, 2.600);
  \coordinate (AA) at (6.41, 2.255);
  \coordinate (c_3) at (7.23, 1.760);
  \coordinate (D) at (-2.00, 3.500);
  \coordinate (J) at (7.00, 1.875);
  \coordinate (V) at (6.75, 1.750);
  \coordinate (AB) at (8.54, 2.605);
  \coordinate (AC) at (8.00, 1.875);
  \coordinate (AD) at (7.75, 2.500);
  \coordinate (B) at (9.22, 2.730);
  \coordinate (AE) at (9.00, 2.625);
  \coordinate (AF) at (8.74, 2.385);
  \coordinate (AH) at (9.50, 2.875);
  
  \coordinate (AI) at (9.25, 3.00);
  \coordinate (AG) at (10.00, 3.00);


\draw[fill=lightgray, opacity = 0.2, draw = none] (-4,8) rectangle (20,7.2);
\draw[fill=lightgray, opacity = 0.2, draw = none] (-4,3) rectangle (20,1.3);
\draw[fill=lightgray, opacity = 0.2, draw = none] (-4,5.5) rectangle (20,4.15);

  \draw (-4, 4.15) -- (20,4.15);
  \draw (-4, 5.5) -- (20,5.5);
  \draw (-4, 7.2) -- (20,7.2);
  
  \draw (A) -- (20,8);
  \draw (C) -- (20,3);
  \draw (E) -- (20,1.3);
  \draw (D) .. controls (0.75, 4.75) and (0.75, 8.00) .. (G);
  \draw[thick] (G) -- (20, 8);

  \node[right] at (20,8) {$P_1'$};
  \node at (1.8, 4.9) {$P_3'^i$};
\draw [ultra thick] (q) -- (H);
\draw [ultra thick] (H) .. controls (0.75, 3.035) and (1.00, 2.000) .. (c_2);
\draw [ultra thick] (c_2) .. controls (3.75, 1.875) and (3.00, 2.625) .. (M);
\draw [ultra thick] (M) .. controls (4.26, 2.705) and (4.35, 2.625) .. (R);
\draw [ultra thick] (R) .. controls (4.75, 2.250) and (4.80, 2.825) .. (U);
\draw [ultra thick] (U) .. controls (5.72, 2.725) and (5.88, 2.600) .. (X);

\coordinate (D_1) at (-1.3, 3.86);
\coordinate (D_2) at (-0.5, 4.54);
\draw [ultra thick] (D) .. controls (D_1) and (D_2) .. (0.3,5.5);
\draw[ultra thick] (0.3,5.5) -- (3,5.5); 
\draw[ultra thick] (3,5.5) .. controls (3.4, 5.2) and (3.7, 4.4) .. (4, 4.15); 
\draw[ultra thick] (4,4.15) -- (7.23, 4.15);
\draw[ultra thick, dashed] (7.23,4.15) -- (11.5, 4.15);

\fill[black] (3.4,5) circle (3px); 
\node at (3.2,4.5) {$c_3$};

\draw[dashed] (7.23,8) -- (7.23, 1.76); 

\draw (q) -- (D);
\draw [ultra thick] (X) .. controls (7.00, 1.875) and (6.75, 1.750) .. (c_3);
\draw [ultra thick] (c_3) .. controls (8.00, 1.875) and (7.75, 2.500) .. (AB);
\draw [ultra thick] (AB) .. controls (9.00, 2.625) and (8.74, 2.385) .. (B);
\draw [ultra thick] (B) .. controls (9.50, 2.875) and (9.25, 3.000) .. (AG);
\draw [ultra thick] (9.98, 3) .. controls (10.65, 3.5) and (11, 3.8) .. (11,3.8);
\draw [ultra thick] (11, 3.8) .. controls (12.65, 5.2) and (14, 5.5) .. (14,5.5);
\draw [ultra thick] (14, 5.5) -- (20, 5.5); 

\node [below,right, font=\scriptsize] at (7.2,7.5) {$d(x,q)=i$};

\node[right] at (20,5.5) {$P_2'$};
\node at (5,6.25) {$A$};
\node at (14,6.25) {$B$};
\node at (5, 3.6) {$C$};

\fill[black] (0.25,5.5) circle (3pt);
\coordinate (D') at (0.6, 5.9);
\fill[black] (D') circle (3pt);
\coordinate (P1) at (1.2,7);
\coordinate (P2) at (2.5,7.99);
\draw [ultra thick] (D') .. controls (P1) and (P2) .. (G);
\node[above,left] at (0.3,6) {$y$};
\draw [ultra thick] (G) -- (20,8);

  \fill[black] (q) circle (3pt);
  \node[below, black] at (q) {$q$};
  \fill[black] (c_3) circle (3pt);
  \fill[black] (D) circle (3pt);
  \fill[black] (2.3,7.6) circle (3pt);
  \node[above, black] at (2.4,8) {$c_1$};

  \node[left, black] at (A) {$P$};
  \node[left, black] at (C) {$Q$};
  \node[left, black] at (-4,5.5) {$Q'$};
  \fill[black] (q) circle (3pt);
  \node at (7.9,1.6) {$c_2$};
  \fill[black] (c_3) circle (3pt);
\end{tikzpicture}

\caption{The second phase of the robber band-reduction strategy. As before, the dashed line at level $i$ indicates where the robber territory is
split into regions $A$, $B$, and $C$.}
\label{fig:second-band-reduction-phase}
\end{figure}
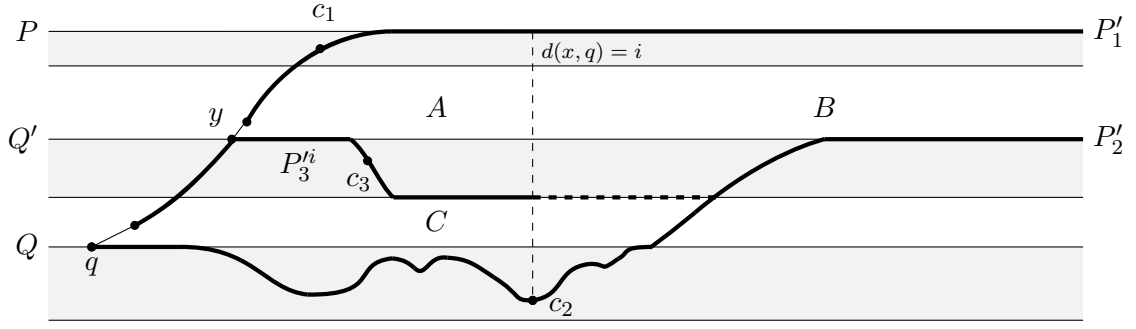

Like in the first phase, we claim that the cops can follow this strategy in a way which allows them to eventually increase the value of $i$, unless $c_2$ reaches $S$ or the robber is caught before. As before, once the cops have established positions as described in the strategies, the cops can play the strategy as long as the robber stays in the same region.

Hence, it remains to show that the cops start this phase in suitable positions, that they can keep following this strategy when the robber moves between regions, and that they can eventually increase the value of $i$.

Since $\ell(c_3) = \ell(r)$ and $q$ is an endpoint of $P_3'^i$, we see that $c_3$ is in the wide shadow of $r$ on $P_3'$. If the robber is between $Q$ and $Q'$, then we arrive at a weakly valid configuration with no robber bands by letting $c_2$ and $c_3$ guard $P_2$ and $P_3'$, respectively. So we may assume that the robber is between $Q'$ and $P$. Since $c_1$ is in the wide shadow of $r$, and $\ell(c_2) > \ell(r)$, all conditions are satisfied for $i = \ell(c_2)$ where the robber is in $A$.

Note that by Lemma \ref{lem:nodoublebypath} there is no vertex which simultaneously lies on a bypath of $P_1'$ and $P_3'$, so if the robber stays in $A$ indefinitely, then Lemma \ref{lem:leisurelyguarding} implies that we can increase $\ell(c_2)$ (and thus increase $i$) after finitely many moves. Hence the robber must move to $B$ eventually.

When the robber moves from $A$ to $B$, the robber's level increases from $i$ to $i+1$. Note that the strategy for $A$ dictates that $\ell(c_3) \in \{i,i-1\}$, so $c_3$ does not have to move on this turn. This means that $c_2$ can move from level $i$ to level $i+1$ (and since $q$ is an endpoint of $P_3'$, this vertex must be in the wide shadow of $r$). By Lemma \ref{lem:nodoublebypath} there is no vertex which simultaneously lies on a bypath of $P_1'$ and $P_2'$, so if the robber stays in $B$ indefinitely, then either the robber moves arbitrarily far away from $q$, forcing $c_2$ to move onto $S$, or we can increase $\ell(c_3)$. The first such increase may happen when $c_2$ is no longer at the same level as the robber, and thus $c_3$ must move up to level $i$, but after at most two increases of $\ell(c_3)$ we can increase $i$. Hence the robber must eventually move back to $A$ or to $C$.

When the robber moves to $A$, the robber's level decreases from $i+1$ to $i$. If $\ell(c_3) = i-1$, then $\ell(c_2) = i+1$ before this move. In this situation, $c_2$ stays at the same vertex and $c_3$ moves to level $i$. If $\ell(c_3) = i$, then $c_2$ can move to (or stay at) level $i+1$. In both cases, we can increase $i$ by 1.

Finally, if the robber moves to $C$, then, as before, $\ell(r)$ decreases from $i+1$ to $i$. Both $c_2$ and $c_3$ can move to level $i$, and since $P_2'$ and $P_3'$ start at $q$, both of cops are in the wide shadow of the robber. We claim that the cops have a strategy to reach a weakly valid configuration with no robber bands.

Note that $QqP_2'$ is a geodesic by Lemma \ref{lem:switchgeodesic}. On each turn $c_2$ moves towards the right boundary of the robber's wide shadow on this geodesic. Note that since $P_2'$ is a subray of $QqP_2'$, the wide shadow on $QqP_2'$ is contained in the wide shadow on $P_2'$. Thus, by moving towards the wide shadow on $QqP_2'$, the cop never leaves the wide shadow on $P_2'$ and therefore keeps guarding $P_2'$ throughout. Since $r$ is confined to a finite set of vertices, $c_2$ eventually reaches the right boundary of the wide shadow of $r$ on $QqP_2'$. Similarly, $c_3$ moves towards the left boundary of the wide shadow on $QqP_3'$ and eventually reaches this left boundary.

Let $T$ be a geodesic from $r$ to $P_2' \cup P_3'$. Since the shortest path from $r$ to $Q$ contains some vertex of $P_2' \cup P_3'$, the length of $T$ is at most $m$. Let $s$ be the other endpoint of $T$. Then $d(q,r) \leq d(q,s) + d(s,r) \leq d(q,s)+ m$. Similarly, $d(\sigma^k(x),r) \leq d(\sigma^k(x),s) + m$. Noting that $s$ lies on a geodesic (either $P_2'$ or $P_3'$) between $q$ and $\sigma^k(x)$, we obtain
\begin{multline*}
     d(q,r) + d(\sigma^k(x),r) \leq d(q,s) + d(\sigma^k(x),s) +2m= d(q,\sigma^k(x)) + 2m\\= d(q,x) + d(x,\sigma^k(x)) + 2m \leq d(x,\sigma^k(x)) + 3m.
\end{multline*}
Since $d(x,\sigma^k(x)) = d(q,\sigma^k(q)) > 3m$ and $d(c_i,r) \leq d(q,r)$, this means that either $c_2$ is between $q$ and $\sigma^k(q)$ on $P_2'$, or $c_3$ is between $x$ and $\sigma^k(x)$ on $P_3'$. Assume without loss of generality the former (the other case is symmetric). Since $c_2$ is at the right boundary of the robber's wide shadow on $QqP_2'$, it follows from Corollary \ref{cor:wideshadowboundary} that $d(c_2,w) = d(r,w)$ for all but finitely many vertices of $Qq$. Consequently, again by Corollary \ref{cor:wideshadowboundary}, $c_2$ is in the wide shadow of $r$ on $QqP_2'\sigma^k(q)Q$ and thus can switch to guarding this geodesic. By restricting to the subpath $P_2'\sigma^k(q)Q$ we have arrived at a weakly valid configuration with no robber bands.
\end{proof}

\bibliographystyle{abbrv}
\bibliography{sources}

\end{document}